\documentclass{amsart}

\usepackage{amsmath}
\usepackage{amsthm,amssymb,color,comment}
\usepackage{mathtools}
\usepackage{bbm}
\usepackage{hyperref}
\usepackage[TS1,T1]{fontenc}
\usepackage{inputenc}
\usepackage{dsfont}
\usepackage{tikz}
\usepackage{enumitem}
\usepackage{soul}
\usepackage[top = 2.5cm, left=2cm, right=2cm, bottom=2.5cm]{geometry}
\usepackage[sort]{cite}

\numberwithin{equation}{section}

\newtheorem{thm}{Theorem}[section]
\newtheorem{prop}[thm]{Proposition}
\newtheorem{lem}[thm]{Lemma}
\newtheorem{cor}[thm]{Corollary}

\usepackage{color}

\newtheorem{Def}[thm]{Definition}
\newtheorem{rem}[thm]{Remark}

\theoremstyle{definition}
\newtheorem{Ass}[thm]{Assumption}

\newcommand{\eq}[1]{\begin{equation}
\begin{split}
#1
\end{split}
\end{equation}}

\newcommand{\lr}[1]{\left( #1 \right)}
\newcommand{\Grad}{\nabla}

\DeclareMathOperator{\DIV}{div}

\newcommand{\R}{\mathbb{R}}
\newcommand{\N}{\mathbb{N}}

\newcommand{\M}{\mathcal{M}}

\newcommand{\diff}{\mathop{}\!\mathrm{d}}

\newcommand{\symm}{\mathbb{R}^{d\times d}_{\text{sym}}}

\newcommand{\wstar}{\overset{\ast}{\rightharpoonup}}

\makeatletter
\newcommand{\doublewidetilde}[1]{{%
  \mathpalette\double@widetilde{#1}%
}}
\newcommand{\double@widetilde}[2]{%
  \sbox\z@{$\m@th#1\widetilde{#2}$}%
  \ht\z@=.9\ht\z@
  \widetilde{\box\z@}%
}
\makeatother

\author{Jakub Woźnicki}
\address{Faculty of Mathematics, Informatics and Mechanics, University of Warsaw, Stefana Banacha 2, 02-097 Warsaw, Poland; Institute of Mathematics of Polish Academy of Sciences, Jana i J\k edrzeja \'Sniadeckich 8, 00-656 Warsaw, Poland}
\email{jw.woznicki@student.uw.edu.pl}
\thanks{Jakub Woźnicki was supported by National Science Center, Poland through project no. 2021/43/O/ST1/03031.}

\author{Ewelina Zatorska}
\address{Mathematics Institute, University of Warwick, Zeeman Building, Coventry CV4 7AL, United Kingdom}
\email{ewelina.zatorska@warwick.ac.uk}
\thanks{The work of E.Z. was  supported by the EPSRC Early Career Fellowship no. EP/V000586/1.}

\begin{document}

\title[Nonlocal Aw--Rascle, Euler alignment]{Existence and weak-strong uniqueness of measure solutions to Euler-alignment/Aw--Rascle--Zhang model of collective behaviour}

\begin{abstract}

We study the multi-dimensional Euler–alignment system with a matrix-valued communication kernel, motivated by models of anticipation dynamics in collective behaviour. A key feature of this system is its formal equivalence to a nonlocal variant of the Aw–Rascle–Zhang (ARZ) traffic model, in which the desired velocity is modified by a nonlocal gradient interaction. We prove the global-in-time existence of measure solutions to both formulations, obtained via a single degenerate pressureless Navier–Stokes approximation. Furthermore, we establish a weak–strong uniqueness principle adapted to the pressureless setting and to nonlocal alignment forces. As a consequence, we rigorously justify the formal correspondence between the nonlocal ARZ and Euler–alignment models: they arise from the same inviscid limit, and the weak–strong uniqueness property ensures that, whenever a classical solution exists, both formulations coincide with it.

\end{abstract}

\keywords{Euler-alignment model, Aw-Rascle-Zhang model, matrix-valued communication weight, measure solutions, weak-strong uniqueness}
\subjclass[2000]{35Q35, 35Q70, 35A02, 35B35}

\maketitle

\section{Introduction}
Among these, the Euler–alignment model, obtained as a hydrodynamic limit of the Cucker–Smale system, plays a central role in describing emergent coordination in crowds, animal flocks, and  swarms. The model consists of pressureless compressible Euler equations with a nonlocal velocity alignment force, and has been the subject of extensive analytical study in recent years; see, for example, \cite{HT08,KMT15,LLSR,LS22,PF24,ChKoWi}.

In this work, we study a multi-dimensional Euler system with a matrix-valued alignment kernel
\begin{align}\label{eq:main_sys_1}
    \left\{\begin{array}{ll}
         \partial_t\rho + \DIV_x(\rho\, u) = 0, \\
         \partial_t(\rho\, u) + \DIV_x(\rho u\otimes u) + \rho\int_{\R^d}D^2K(x-y) (u(x) - u(y))\rho(y)\diff y = 0,
    \end{array} \right. \quad \text{ in }(0, T)\times\R^d,
\end{align}
where 
\eq{\label{khess}
D^2 K = \lr{\partial^2_{x_i x_j}K }_{i,j=1}^n.}
Such Hessian-type kernels were introduced by Shu and Tadmor \cite{ShuTadmor21}, to model {\emph{anticipation dynamics}}, and lead to flocking behavior under attractive potentials.

A distinctive feature of \eqref{eq:main_sys_1} is that it also admits a hydrodynamic formulation reminiscent of the Aw–Rascle–Zhang (ARZ) traffic model. Formally introducing a velocity field $w$ via
\eq{\label{wu}
w=u+ \Grad K(x)\star\rho,}
yields the nonlocal multi-dimensional ARZ system
\begin{align}\label{eq:main_sys_rho_w}
    \left\{\begin{array}{ll}
         \partial_t\rho + \DIV_x(\rho\, u) = 0, \\
         \partial_t(\rho\, w) + \DIV_x(\rho w\otimes u) = 0,
    \end{array} \right. \quad \text{ in }(0, T)\times\R^d.
\end{align}

The one-dimensional ARZ model is one of the most established macroscopic models for vehicular traffic flow \cite{AR2000, Zhang2002, GP2006, Goatin2006}. It consists of conservation laws \eqref{eq:main_sys_rho_w} for the density $\rho$ and the momentum associated with the desired velocity $w$. The difference $w-u$ is referred to as the {\emph{pressure}} or {\emph{offset}} and is typically modeled as a nonnegative, increasing function of $\rho$, such as $p(\rho)=\rho^\gamma$. The quantity $p(\rho)$ plays an analogous role  to pressure in classical fluid dynamics—it governs the propagation of congestion waves—but in contrast to fluid models, it modifies the entire transport structure of the ARZ system, not only the forcing term.

Recently,   generalizations and non-local variants of ARZ model have been introduced, see \cite{FaHeSe, HT2025, Marconi2025, Amadori2025}, motivated both by traffic flow modeling and by analogies with collective motion systems. In these models,  the velocity takes the form $u=V(w,\rho)$, capturing local or nonlocal interactions. In particular, \cite{Amadori2025} studied the case $u = w - \xi\star \rho$ for a smooth kernel $\xi$, and observed that this nonlocal ARZ system is equivalent to the one-dimensional Euler–alignment model \eqref{eq:main_sys_1} with communication weight $\Psi=D^2K$. Another, but this time a multi-dimensional, extension of ARZ appears in \cite{ABDM24}, where  system \eqref{eq:main_sys_rho_w} is suplemented with the offset relation
$$w=u+\Grad(\rho^{-1}-\rho_{max}^{-1})^{-\gamma},$$
for some maximal packing density $\rho_{max}$. The congestion term in form of a gradient, results in a singular diffusion in the continuity equation (when written in terms of $\rho$ and $w$) to enforce capacity constraints in the crowd density while inducing a steering behaviour typical for the flow of pedestrians.
Incorporating the nonlocal term \eqref{wu} in this setting corresponds exactly to the anticipation-driven dynamics of \cite{ShuTadmor21}.

The primary goal of this paper is to establish global-in-time measure solutions to \eqref{eq:main_sys_1} and to the ARZ formulation \eqref{wu}–\eqref{eq:main_sys_rho_w}, without smallness assumptions on the initial data. We also prove a weak–strong uniqueness principle, showing that these measure solutions coincide with classical solutions whenever the latter exist.

The equivalence between \eqref{eq:main_sys_1} and \eqref{wu}–\eqref{eq:main_sys_rho_w} was recently explored in \cite{CPSZ24} at the microscopic, mesoscopic, and macroscopic levels. We also refer to \cite{CPSZ24} for a useful overview of other results on these models, including existence and uniqueness results as well as their derivation through mean-filed limit, based on the previous results of Kim \cite{Kim22} Peszek and Poyato \cite{PP23, PP24}  Fabisiak and Peszek \cite{PF24}.
 In particular, under the assumption that $K$ is $\lambda$-convex and weakly singular, it was shown that certain class of measure-valued solutions to the mesoscopic (kinetic) formulation is a mean-field limit of atomic solutions. Under stronger assumptions on the kernel ($D^2K\in W^{1,\infty}(\mathbb{R}^d)$) the convergence of empirical measures towards the classical but short-time solutions to \eqref{eq:main_sys_rho_w} was shown following the approach of Carrillo and Choi \cite{CC21}.  Concerning the connection between the macroscopic hydrodynamic formulation \eqref{eq:main_sys_rho_w} and the  kinetic system, the authors were only able to justify that a suitably regular solution to \eqref{eq:main_sys_rho_w} generates mono-kinetic solution to the kinetic system. However, global solutions are only known to exists under smallness assumptions; see \cite{CTT21, Shvydkoy, Danchin,TT14,CCKT}, and references therein.

For the local version of the model \eqref{wu}–\eqref{eq:main_sys_rho_w}, i.e. when 
$K(x)\approx \delta(x)$, the existence of measure-valued solutions and a weak–strong uniqueness principle were established by Chaudhuri, Gwiazda, and Zatorska in \cite{CGZ22}; see also \cite{Chaudhuri_Piasecki_Zatorska_2025} for local existence of strong solutions. The existence of infinitely many global-in-time weak solutions for bounded initial data, obtained via the convex integration method, was demonstrated in \cite{CFZ24}.

The same convex integration technique was already used in \cite{CFGG17} to construct weak solutions to the Euler system with general nonlocal interactions, including attraction–repulsion and alignment. In particular, they treated a pressureless nonlocal Euler system and exploited the weak repulsive Poisson interaction to prove weak–strong uniqueness. It is worth noting that all of these results were obtained in bounded or periodic domains, whereas our analysis is carried out in the whole space. We also mention the works \cite{BrMa19, CWZ20} on inviscid limits of Navier–Stokes-type approximations for nonlocal Euler systems, with and without pressure. In \cite{ChKoWi}, dissipative measure-valued solutions for the Euler–alignment model were studied and weak–strong uniqueness was proved. However, in all of these works a crucial role is played by estimates in Lebesgue or (dual) Sobolev norms for the density. Since systems \eqref{eq:main_sys_rho_w} and \eqref{eq:main_sys_1} lack a classical pressure term, such techniques are not  applicable here.

In this paper, we first establish global-in-time existence of measure solutions to \eqref{eq:main_sys_1} (and equivalently to the system \eqref{wu}-\eqref{eq:main_sys_rho_w}). Instead of approximating \eqref{eq:main_sys_1} via its microscopic or kinetic formulations—as in, for instance, \cite{mucha2018thecucker}, where measure-valued solutions to the Cucker–Smale model with singular scalar weights were obtained—we construct solutions by a vanishing viscosity limit of a suitably designed degenerate pressureless Navier–Stokes approximation. In contrast with \cite{BrMa19} and \cite{CGZ22}, our approach relies on the existence theory developed by Vasseur and Yu \cite{VaYu2016} and on its nonlocal generalization in \cite{mucha2025construction}.

The key advantage of this approximation is that density-dependent viscosity yields enhanced regularity of the density, thanks to the Bresch–Desjardins entropy estimate \cite{bresch2003existence}, while the Mellet–Vasseur estimate \cite{mellet2007onthebarotropic} gives compactness for the velocity field. Unlike the case of constant viscosity, this allows one to construct weak solutions even in the absence of pressure, as in \cite{mucha2025construction}, which is essential for our problem. Indeed, in the energy identity the dissipative term $a\DIV_x(\nu\mathbb{D}u)$ appears as $a\nu\int|\mathbb{D}u|^2\diff x$, and thus it vanishes in the weak formulation as $a\to 0^+$, but for pressure terms $a\nabla_xp(\rho)$ this is not true. Our proof crucially exploits this structure. Moreover, we use natural bounds in the space of Radon measures, obtained from the continuity equation and conservation of momentum, to pass to the limit in all terms except the convective term, which may generate a concentration defect. Adapting techniques from \cite{abbatiello2020onaclass, woznicki2022weak}, we control this defect measure via the energy inequality.

Our second main contribution is a weak–strong uniqueness principle, based on a relative entropy method adapted to the pressureless setting. This method originates from Dafermos \cite{dafermos1979thesecond, dafermos2016hyperbolic} for scalar conservation laws, and was extended to incompressible Euler equations by Brenier et al. \cite{brenier2011weak}. 
It has since been applied to stability, asymptotic limits, and dimension reduction problems \cite{bella2014dimension, christoforou2018relative, giesselemann2017stability, feireisl2012relative}. In our case, instead of relying on convex entropy associated with pressure potentials, the method employs Wasserstein-2 distance used before by Figali and Kang in \cite{figali2019arigorous} and by Carrillo and Choi \cite{CC21} for the derivation of mono-kinetic Euler-alignment model from the particle and the kinetic levels, respectively. Consequently, classical solutions from \cite{CC21, CPSZ24} are shown to be unique within the class of measure solutions.
As a corollary, our results rigorously justify the formal correspondence between the generalized nonlocal ARZ system \eqref{wu}-\eqref{eq:main_sys_rho_w} and the Euler–alignment system \eqref{eq:main_sys_1}, previously observed in \cite{Amadori2025, CPSZ24}. We establish this equivalence in two ways: first, by proving that both systems arise as inviscid limits of the same approximation; and second, by showing that whenever a classical solution exists, both formulations coincide through our weak–strong uniqueness principle.

The rest of the paper is organized as follows. In Section \ref{Sec:2} we explain our notation, introduce assumptions on the data and on the  kernel $K$, define our measure solution and formulate our main results -- Theorems  \ref{thm:main_existence} and \ref{thm:stability_initial_datum}. Then, in Section \ref{Sec:3} we recall and prove some compactness tools in measure spaces, used throughout the paper. In Section \ref{Sec:4} we prove Theorem \ref{thm:main_existence} starting from a suitable approximation which is further discussed in  the Appendix \ref{appendix:existence_for_approximation}. Finally, in Section \ref{Sec:5} we prove the weak-strong uniqueness via a relative entropy argument, concluding the proof of Theorem \ref{thm:stability_initial_datum}.

\section{Notation, assumptions and the main results}\label{Sec:2}

First, we introduce the notation used through the work. Let $d$ be the dimension of the space, and $T > 0$ denote the length of time interest. We will write $x$ for an element of $\R^d$ and $t$ for an element of $(0,T)$. Next, for any vectors $a,b\in \mathbb{R}^d$ we write $a\cdot b$ for the standard scalar product of $a$ and $b$. Similarly, the space $\symm$ denotes the space of symmetric $d\times d$ matrices and for any $A,B \in \symm$ we denote the scalar product by $A:B$, that is $A : B := \mathrm{tr}(A^T B)$. Moreover, the symbol $\otimes$ is reserved for the tensor product, that is, whenever $a,b\in \mathbb{R}^d$ we denote by $a\otimes b \in \symm$ as $(a\otimes b)_{ij}:=a_ib_j$ for $i,j=1,\ldots, d$. We use the standard notation for, continuous (where $C_c$ denotes compactly supported functions and $C_b$ continuous and bounded ones), Sobolev and Lebesgue function spaces, as well as the space of Radon measures (where $\mathcal{M}^+$ denotes the space of positive measures, and $\mathcal{P}$ the space of probability measures) and frequently do not distinguish between scalar-, vector- or matrix-valued functions.
In addition, to shorten the notation, we sometimes use the following simplifications. When $f \in L^p(\R^d)$, we simplify it to $f \in L^p_x$, and if $f \in L^p(0,T; L^q(\R^d))$, then we write $f \in L^p_t L^q_x$.

Symbol $\mathbb{D}u$ denotes the symmetric part of the spatial gradient $\nabla_x$ of a function $u$, i.e. $\mathbb{D}u = \left(\nabla_x u + (\nabla_x u)^T\right)/2$. Throughout the paper we also employ the universal constant $C$ that may vary from line to line, but depends only on data.

Now, let us introduce the assumptions on the initial datum.

\begin{Ass}[Assumptions on initial data]\label{ass:initial_data}
We assume that $u_0\in L^2(\R^d;\diff\rho_0)$, $\rho_0\in \mathcal{P}(\R^d)$, and
        $$
            \int_{\R^d}|x|^2\diff\rho_0 < +\infty.
        $$
\end{Ass}

\begin{Ass}[Assumptions on the kernel]\label{ass:K}
Depending on the result, a subset of the following assumptions on the kernel will be made:
\begin{enumerate}
    \item[(i)] $D^2K \in C_b(\R^d)$,
    \item[(ii)] $D^2K$ is even and positive semi-definite,
    \item[(iii)] $D^2K$ is Lipschitz.
\end{enumerate}
\end{Ass}

In order to define the notion of measure solutions we need to first to recall definition of a flat metric on the space of Radon measures (discussion on its properties is postponed to Section \ref{Sec:3}).
\begin{Def}\label{def:flat_metric}
    Let $\mu, \nu\in \mathcal{M}(\R^d)$. Then, we may define a flat metric
    \begin{align*}
        d_f(\mu, \nu) = \sup\left\{\int_{\R^d}\varphi\diff(\mu - \nu)\,\Big|\,\varphi\in C^1(\R^d),\,\|\varphi\|_\infty\leq 1,\,\|\nabla_x\varphi\|_\infty\leq 1\right\}.
    \end{align*}
\end{Def}

The measure solutions to problem \eqref{eq:main_sys_1} are defined as follows.

\begin{Def}\label{def:measure_solution}
    We say that the triple $(\rho, u, \mu)$ is a measure solution to the system \eqref{eq:main_sys_1} with $K$ satisfying Assumption \ref{ass:K} (i) iff  the following hold:
      \begin{enumerate}
      \item $(\rho,u,\mu)$ belongs to the regularity class:
        \begin{align*}
        \rho\in C([0, T]; (\mathcal{P}(\R^d), d_f)),\quad u\in L^\infty(0, T;& L^2(\diff\rho_t)),\quad \mu\in L^\infty(0, T; \M^+(\R^d; \R^{d\times d}_{\mathrm{sym}})).\\
        \sup_{t\in (0, T)}\int_{\R^d}|x|^2\diff\rho_t(x) < +\infty,&\qquad\sup_{t\in (0, T)}\int_{\R^d}|x||u(t,x)|\diff\rho_t(x) < +\infty.
        \end{align*}

        \item The continuity equation
        \begin{align}\label{eq:final_mass_conservation}
        \int_{\R^d}\phi(t, x)\diff\rho_t(x) - \int_{\R^d}\phi(0, x)\diff\rho_0(x) = \int_0^t\int_{\R^d}\partial_t\phi(\tau,x) + \nabla_x\phi(\tau, x)\cdot u(\tau, x)\diff\rho_\tau(x)\diff\tau,
        \end{align}
       holds for every $t\in [0, T]$ and $\phi\in C^1([0, T]\times \R^d)$, such that $|\phi(t, x)|, |\partial_t\phi(t, x)|\leq C(1 + |x|^2)$, $|\nabla_x\phi(t, x)|\leq C(1 + |x|)$.
        \item The momentum equation
        \begin{equation}\label{eq:final_momentum_conservation}
            \begin{split}
                &\int_{\R^d}\phi(t, x)\cdot u(t, x)\diff\rho_t(x) - \int_{\R^d}\phi(0, x)\cdot u_0(x)\diff\rho_0(x)\\
                &= \int_0^t\int_{\R^d}\partial_t\phi(\tau, x)\cdot u(\tau, x) + u(\tau, x)\otimes u(\tau, x) : \nabla_x\phi(\tau, x)\diff \rho_\tau(x)\diff \tau\\
                &\quad+ \int_0^t\int_{\R^d}\nabla_x\phi(\tau, x) : \diff\mu_\tau(x)\diff \tau + \int_0^t\int_{\R^{2d}}\phi(\tau, x)D^2K(x - y) u(\tau, y)\diff\rho_\tau(y)\diff\rho_\tau(x)\diff \tau\\
                &\quad- \int_0^t\int_{\R^{2d}}\phi(\tau, x) D^2K(x-y)u(\tau, x)\diff\rho_\tau(y)\diff\rho_\tau(x)\diff\tau,
            \end{split}
        \end{equation}
        holds for a.e. $t\in (0, T)$ and $\phi\in C^1([0, T)\times \R^d)$, such that $|\phi(t, x)|, |\partial_t\phi(t, x)|\leq C(1 + |x|)$, $|\nabla_x\phi(t, x)|\leq C$.
        \item The energy inequality
        \begin{equation}\label{ineq:final_energy_inequality}
            \begin{split}
                \frac{1}{2}\int_{\R^d}|u(t, x)|^2\diff\rho_t(x) + \frac{1}{2}\mathrm{tr}(\mu_t(\R^d)) \leq e^{4t\|D^2K\|_\infty}\frac{1}{2}\int_{\R^d}|u_0|^2\diff\rho_0,
            \end{split}
        \end{equation}
        holds for a.e. $t\in (0, T)$. 
        
        If in addition $K$ satisfies Assumption \ref{ass:K} (ii), then the energy inequality is uniform with respect to time, i.e.  
        \begin{equation}\label{ineq:energy_inequality_for_symmetric}
            \begin{split}
                &\frac{1}{2}\int_{\R^d}|u(t, x)|^2\diff\rho_t(x) + \frac{1}{2}\mathrm{tr}(\mu_t(\R^d))\\
                &\quad+ \frac{1}{2}\int_0^t\int_{\R^{d\times d}}(u(s, x) - u(s, y))D^2K(x - y)(u(s, x) - u(s, y))\diff \rho_s(x)\diff \rho_s(y)\diff s \leq \frac{1}{2}\int_{\R^d}|u_0|^2\diff\rho_0,
            \end{split}
        \end{equation}
        for a.e. $t\in(0, T)$.
    \end{enumerate}
  
\end{Def}

\begin{Def}\label{def:measure_solution_rho_w}
    We say that the quadruple $(\rho, w, u, \nu)$ is a measure solution to system \eqref{wu}, \eqref{eq:main_sys_rho_w} iff $\rho,u$ belong to the regularity class specified in Definition \ref{def:measure_solution} and satisfy the continuity equation in the same  of \eqref{eq:final_mass_conservation}. Moreover, for $w$ defined via
\begin{align}\label{eq:final_compatibility_rho_w}
        w(t,x)\diff\rho_t := u(t, x)\diff\rho_t + (\nabla_xK\star\diff\rho_t)\diff\rho_t\qquad \text{ a.e. }t\in(0, T).
    \end{align}
    Then, there exists a matrix-valued measure $\nu\in L^\infty(0, T; \M(\R^d; \R^{d\times d}))$, such that the momentum equation \eqref{eq:main_sys_rho_w}$_2$
        \begin{equation}\label{eq:final_momentum_conservation_rho_w}
            \begin{split}
                &\int_{\R^d}\phi(t, x)\cdot w(t, x)\diff\rho_t(x) - \int_{\R^d}\phi(0, x)\cdot w_0(x)\diff\rho_0(x)\\
                &= \int_0^t\int_{\R^d}\partial_t\phi(\tau, x)\cdot w(\tau, x) + w(\tau, x)\otimes u(\tau, x) : \nabla_x\phi(\tau, x)\diff \rho_\tau(x)\diff \tau+ \int_0^t\int_{\R^d}\nabla_x\phi(\tau, x) : \diff\nu_\tau(x)\diff \tau
            \end{split}
        \end{equation}
        holds for a.e. $t\in (0, T)$ and $\phi\in C^1([0, T)\times \R^d)$, for which $|\phi(t, x)|, |\partial_t\phi(t, x)|\leq C(1 + |x|)$, $|\nabla_x\phi(t, x)|\leq C$.
    Furthermore,
 \begin{align*}
        \sup_{t\in (0, T)}\int_{\R^d}|x||w(t,x)|\diff\rho_t(x) < +\infty,
 \end{align*}   
and for a.e. $t\in (0, T)$
        \begin{equation}\label{ineq:final_energy_inequality_rho_w}
            \begin{split}
                \frac{1}{2}\int_{\R^d}|w(t, x)|^2\diff\rho_t(x) + \frac{1}{2}\int_{\R^d}|u(t, x)|^2\diff\rho_t(x) + \mathrm{tr}(\nu_t(\R^d))\\
                \leq C(\|D^2K\|_\infty, \||u_0|^2\|_{L^1(\diff\rho_0)}, \||x|^2\|_{L^1(\diff\rho_0)}).
            \end{split}
        \end{equation} 

\end{Def}

Finally, we state our  first main result.

\begin{thm}\label{thm:main_existence}
    Let the initial data $(\rho_0, u_0)$ satisfy Assumption \ref{ass:initial_data}, and let $K$ satisfy Assumption \ref{ass:K} (i). 
    \begin{enumerate}
    \item Then, there exists a measure solution $(\rho, u, \mu)$ to \eqref{eq:main_sys_1} in the sense of Definition \ref{def:measure_solution}. Moreover, this solution is obtained as the vanishing-viscosity limit of weak solutions to the compressible Navier--Stokes system with density-dependent viscosity
    
  \item For $w$  defined via \eqref{eq:final_compatibility_rho_w}, 
 the quadruple $(\rho, w, u, \nu)$ is a measure solution to system \eqref{wu}-\eqref{eq:main_sys_rho_w} in the sense of Definition \ref{def:measure_solution_rho_w}.
        \end{enumerate}
\end{thm}

Our second result concerns the weak-strong uniqueness property of the solutions given by Theorem \ref{thm:main_existence}. As always, such a result shows that the notion of solution given by Definition \ref{def:measure_solution} makes sense, as it reduces to a strong solution whenever it exists. We note here that in fact, there is an already existing theory of strong solutions to the system \eqref{eq:main_sys_1}. One can find relevant information in \cite[Theorem 4.1]{CC21}, and \cite[Theorem 3.1]{ha2014ahydrodynamic}, where authors prove the existence of strong solutions, whenever the initial data and communication kernel are smooth enough, and either there is some smallness of norms of initial datum or the existence is only local in time. In principle, one should require the kernel to be even and satisfy $D^2K\in W^{1, p}(\R^d)$ for any $p\in [1, +\infty]$. 

\begin{thm}\label{thm:stability_initial_datum} Suppose, that the kernel K satisfies the Assumptions \ref{ass:K} (i)--(iii). Assume, that there exists a strong solution $(r, v)\in C([0, T]; \mathcal{P}(\R^d)) \times L^\infty(0, T; W^{1,\infty}(\R^d))$, $r > 0$ on $[0, T]\times\R^d$ to the system
    \begin{align}\label{eq:main_sys_strong}
    \left\{\begin{array}{ll}
         \partial_t r + \DIV_x(r\, v) = 0, \\
         \partial_tv + (v\cdot\nabla_x)v + \int_{\R^d}D^2K(x-y) (v(x) - v(y))\diff r(y) = 0,
    \end{array} \right. \quad \text{ in }(0, T)\times\R^d,
    \end{align}
    with the initial datum $(r_0,v_0)$.
    Let $(\rho^n_0, u^n_0)$ satisfy Assumption \ref{ass:initial_data} as well as 
    \begin{align*}
        \|\rho^n_0 - r_0\|_{TV}&\rightarrow 0,\\
        \||x|^2(\rho_0^n - r_0)\|_{TV}&\rightarrow 0,\\
        \int_{\R^d}|u^n_0 - v_0|^2\diff\rho_0^n&\rightarrow 0.
    \end{align*}
  Then, the sequence of measure solutions $(\rho^n, u^n, \mu^n)$ with initial data $(\rho^n_0, u^n_0)$, obtained in Theorem \ref{thm:main_existence}, satisfies the inequality
   \begin{equation*}
        \begin{split}
    &\int_{\R^d}|v(t, x) - u^n(t, x)|^2\diff\rho^n_t(x) + \mathrm{tr}(\mu^n_t(\R^d)) + W_2^2(\rho^n_t, r_t)\\
    &\qquad\leq e^{C(T, \|D^2K\|_{W^{1,\infty}},\|v\|_{W^{1,\infty}})}\left(\int_{\R^d}|u_0^n - v_0|^2\diff\rho_0^n(x)+ \|\rho^n_0 - \rho_0\|_{TV} + \||x|^2(\rho_0^n - \rho_0)\|_{TV}\right).
        \end{split}
    \end{equation*}
    In particular, the measure solutions $(\rho^n, u^n, \mu^n)$ converge to the strong solutions $(r, v)$ in the sense that
    \begin{align*}
        &\sup_{t\in (0, T)}W_2(\rho^n_t, r_t) \rightarrow 0,\\
        &\sup_{t\in (0, T)}\int_{\R^d}|u^n(t, x) - v(t, x)|^2\diff\rho^n_t(x) \rightarrow 0,\\
        &\sup_{t\in (0, T)}\|\mu^n\|_{TV}\rightarrow 0.
    \end{align*}
\end{thm}

\begin{cor}\label{cor:weak-strong-uniq}
Suppose that for given initial conditions $(\rho_0, u_0)$ satisfying Assumption \ref{ass:initial_data}, there exists a strong solution $(r, v)\in C([0, T]; \mathcal{P}(\R^d)) \times L^\infty(0, T; W^{1,\infty}(\R^d))$, $r > 0$ on $[0, T]\times\R^d$ satisfying
    \begin{align}\label{eq:main_sys_strong_weak_strong_argument}
    \left\{\begin{array}{ll}
         \partial_t r + \DIV_x(r\, v) = 0, \\
         \partial_tv + (v\cdot\nabla_x)v + \int_{\R^d}D^2K(x-y) (v(x) - v(y))\diff r(y) = 0,
    \end{array} \right. \quad \text{ in }(0, T)\times\R^d.
    \end{align}
    Suppose moreover that the kenrel K satisfies Assumption \ref{ass:K} (i)-(iii). Then, the measure solution $(\rho, u, \mu)$ arising from the given initial conditions coincides with the strong solutions, in the sense that
    $$
    \diff\rho_t(x) = \diff r_t(x) \text{ for all }t\in [0, T],\quad u(t, x) = v(t, x)\text{ a.e. in }[0, T]\times\R^d, \quad \mu\equiv 0.
    $$
\end{cor}

\begin{cor} 
Suppose that for given initial conditions $(r_0, v_0)$ satisfying Assumption \ref{ass:initial_data}, there exists a strong solution $(r, v)\in C([0, T]; \mathcal{P}(\R^d)) \times L^\infty(0, T; W^{1,\infty}(\R^d))$, $r > 0$ on $[0, T]\times\R^d$ satisfying \eqref{eq:main_sys_strong_weak_strong_argument}.
 Define
    \begin{align}\label{eq:varpi}
        \varpi:= v + \nabla_xK\star r.
    \end{align}
Then   $(r, v,\varpi)$ is a strong solution to
    \begin{align}\label{eq:main_sys_strong_2}
    \left\{\begin{array}{ll}
         \partial_t r + \DIV_x(r\, v) = 0, \\
         \partial_t(r\,\varpi) + \DIV_x(r\,\varpi\otimes v) = 0,
    \end{array} \right. \quad \text{ in }(0, T)\times\R^d.
    \end{align}
    Suppose, moreover, that $K$ satisfies Assumptions \ref{ass:K} (i)--(iii). Then,  $w$ defined in \eqref{eq:final_compatibility_rho_w} coincides with $\varpi$, in the sense that
    $$
   \,w(t, x) = \varpi(t, x)\text{ a.e. in }(0, T)\times\R^d.
    $$
\end{cor}
\begin{proof}
    For strong solutions, equations \eqref{eq:main_sys_strong_2} and \eqref{eq:main_sys_strong_weak_strong_argument} are equivalent for $\varpi$ defined by \eqref{eq:varpi}. Thus, by Corollary \ref{cor:weak-strong-uniq}, we know that 
    $$
    \diff\rho_t(x) = \diff r_t(x) \text{ for all }t\in [0, T],\quad u(t, x) = v(t, x)\text{ a.e. in }[0, T]\times\R^d, \quad \mu\equiv 0.
    $$
    Then, from the equation \eqref{eq:final_compatibility_rho_w}
    \begin{align*}
    &r(t,x)\,w(t, x)\diff x = w(t, x)\diff\rho_t = u(t, x)\diff\rho_t + (\nabla_xK\star\diff\rho_t)\diff\rho_t\\
    &= r(t, x)\, v(t, x)\diff x + (\nabla_x K\star r)r\diff x = r(t, x)\,\varpi(t, x)\diff x.
    \end{align*}
    As $r > 0$ the equation above implies
    $$
    w(t, x) = \varpi(t, x),\quad \text{a. e. in }(0, T)\times\R^d.
    $$

\end{proof}

\section{
Basic definitions and auxiliary proposition in the space of measures
}\label{Sec:3}
 The following section is devoted to the introduction of basic concepts in the space of Radon measures, and necessary propositions for the proofs of main results. For more information, we refer our readers to the books \cite{Ambrosio2000, evans2015measure, Dull2021spacesofmeasures, Santambrogio}. We begin our discussion with establishing the definitions that we will use for weak* and weak convergence of measures.

 \begin{Def}
     Let $\{\mu_n\}_{n\in\N}\subset\mathcal{M}(\R^d)$ be a sequence of Radon measures. We say that
    $$
    \mu_n \wstar \mu, \text{ weakly* in }\mathcal{M}(\R^d),
    $$
    whenever
    $$
    \int_{\R^d}f(x)\diff\mu_n(x) \rightarrow \int_{\R^d}f(x)\diff\mu(x)\text{, for every }f\in C_0(\R^d),
    $$
    and we say that
    $$
    \mu_n \rightharpoonup \mu, \text{ weakly in }\mathcal{M}(\R^d),
    $$
    if
    $$
    \int_{\R^d}f(x)\diff\mu_n(x) \rightarrow \int_{\R^d}f(x)\diff\mu(x)\text{, for every }f\in C_b(\R^d).
    $$
\end{Def}

The following propositions explain the connection between weak* convergence of measures and lower-, upper- semicontiuity properties of measures and their variations.

\begin{prop}[Theorem 1, p. 54, \cite{evans2015measure}]\label{prop:portmanteau_lemma}
    Let $\{\mu_n\}_{n\in\N}\subset \M^+(\R^d)$ be a sequence of measures such that
    $$
        \mu_n \wstar \mu\text{ weakly* in }\M(\R^d).
    $$
    Then, for any open $U\subset\R^d$ and compact $K\subset\R^d$
    \begin{align*}
        \liminf_{n\to+\infty}\mu_n(U) \geq \mu(U),\qquad \limsup_{n\to+\infty}\mu_n(K) \leq \mu(K).
    \end{align*}
    Moreover
    \begin{align*}
        \lim_{n\to +\infty}\mu_n(B) = \mu(B), \text{ for any Borel B, such that }\mu(\partial B) = 0.
    \end{align*}
\end{prop}

\begin{prop}[Proposition 1.62, \cite{Ambrosio2000}]\label{prop:comparison_between_variations}
    Let $\{\mu_n\}_{n\in\N}\subset\M(\R^d)$ be such that
    $$
    \mu_n \wstar \mu \text{ weakly* in }\M(\R^d),
    $$
    and
    $$
    |\mu_n|\wstar \lambda \text{ weakly* in }\M(\R^d),
    $$
    then $\lambda\geq |\mu|$.
\end{prop}

\begin{cor}\label{cor:portmanteau_for_variation}
    From Proposition \ref{prop:portmanteau_lemma} and Proposition \ref{prop:comparison_between_variations} if
    $$
            \mu_n\wstar \mu \text{ weakly* in }\M(\R^d),
    $$
    then for any open $U\subset\R^d$
    $$
    \liminf_{n\to+\infty}|\mu_n|(U) \geq |\mu|(U).
    $$
\end{cor}

Now we move to the definition of a certain metric on the space of Radon measures, which is crucial in our analysis. 
It turns out, that in some cases this metric can allow us to metrize the weak* convergence.

\begin{Def}
    We call the family of measures $\{\mu_\alpha\}\subset\M(\R^d)$ tight, if and only if for every $\varepsilon > 0$, there exists a compact set $K$ such that for any $\alpha$
    $$
    |\mu_\alpha|(\R^d\setminus K) < \varepsilon.
    $$
\end{Def}

\begin{rem}
    Since any two single Radon measures are tight, it is enough to take the supremum over the function $\varphi\in C^1_c(\R^d)$ in the Definition \ref{def:flat_metric}.
\end{rem}

\begin{prop}[Theorem 2.7, \cite{gwiazda2010anonlinear}]\label{prop:tight_sequence_equivalences}
    For any tight sequence $\{\mu_n\}_{n\in\N}\subset\M(\R^d)$ the following equivalences hold
    \begin{enumerate}
        \item \begin{align*}
            \mu_n \wstar \mu \text{ weakly* in }\M(\R^d)\Longleftrightarrow \left\{\begin{array}{ll}
                 \lim_{n\to +\infty}d_f(\mu_n, \mu) \rightarrow 0 \\
                 \sup_{n\in\N}|\mu_n|(\R^d) < +\infty.
            \end{array}
            \right.
        \end{align*}
        \item For any threshold $r > 0$, the set $\mathcal{K}\subset \{\mu\in \M(\R^d)\,|\,|\mu|(\R^d) \leq r\}$ is relatively compact with respect to the flat metric $d_f$ if the set $\mathcal{K}$ is tight.
    \end{enumerate}
\end{prop}

\begin{prop}[Proposition 6.1, \cite{chaudhuri2024existence}]\label{prop:lowersemicont_measures}
    Let $g:\R\rightarrow [0, +\infty)$ be a superlinear, convex, and lower semicontinuous function. Then, the functional
    $$
    \M^+(\R^d)\times \M(\R^d)\ni(\mu, \nu) \mapsto\mathcal{G}(\mu, \nu) = \left\{\begin{array}{ll}
         \int_{\R^d}g\left(\frac{\diff\nu}{\diff\mu}\right)\diff \mu,\text{ whenever }\nu << \mu  \\
         +\infty \text{ otherwise,}
    \end{array}\right.
    $$
    is lower semicontinuous with respect to the weak* convergence of measures.
\end{prop}

Next, we look at an imporant lemma, which allows us to differentiate measures, after we weakly converge with them.

\begin{lem}\label{lem:differentiation_of_limit_measure}
    Let $\{\nu_n\}_{n\in\N}\subset\M(\R^d)$ and $\{\mu_n\}_{n\in\N}\subset\M^+(\R^d)$ be the sequences of finite Radon measures such that
    $$
        \nu_n << \mu_n\text{ for every }n\in\N.
    $$
    Moreover, suppose that there exists a non-decreasing continuous function $f:[0, +\infty)\rightarrow [0, +\infty)$, $f(0) = 0$ such that
    $$
    |\nu_n|(A) \leq f(\mu_n(A)) \text{ for any measurable sets }A,\,n\in\N.
    $$
    and
    \begin{align*}
        \nu_n &\rightharpoonup \nu \text{ weakly* in }\M(\R^d),\\
        \mu_n &\rightharpoonup \mu\text{ weakly* in }\M^+(\R^d).
    \end{align*}
    Then,
    $$
        \nu << \mu.
    $$
\end{lem}
\begin{proof}
    \underline{Step 1: reduction to compact sets.} Suppose that we know that
    $$
    \mu(K) = 0 \Rightarrow |\nu|(K) = 0, \text{ for every compact set }K.
    $$
    Let $A$ be an arbitrary Borel set such that $\mu(A) = 0$. Then, for any $K\subset A$, which is compact, $\mu(K) = 0$, and by our assumption also $|\nu|(K) = 0$. Hence, by inner regularity
    $$
    |\nu|(A) = \sup\{|\nu|(K)\,|\, K\subset A, K\text{ compact}\} = 0.
    $$
    \underline{Step 2: proof for compact sets.} Fix a compact set $K\subset\R^d$ such that $\mu(K) = 0$. Now for any $x\in K$ fix a radius $r_x$, for which $\mu(\partial B(x, r_x)) = 0$. Indeed we can do that, since $\partial B(x, r) \cap \partial B(x, s) = \emptyset$ for $r\neq s$, there can be only countably many sets of the form $\partial B(x, r)$ of positive measure. In fact, $r_x$ can be arbitrarily small. Thus
    $$
    K\subset \bigcup_{x\in K}B(x, r_x).
    $$
    Since $K$ is compact, we may find a finite cover such that
    $$
    K\subset \bigcup_{k = 1}^N B(x_k, r_k).
    $$
    By inner regularity for the sets $B(x_k, r_k)\setminus K$, for fixed $\varepsilon > 0$, we may find a compact set $C_k\subset B(x_k, r_k)\setminus K$ for which
    $$
    \mu((B(x_k, r_k) \setminus K)\setminus C_k) < \frac{\varepsilon}{N}.
    $$
    Again, as $C_k$ is compact, and $B(x_k, r_k) \setminus K$ is open, similarly as for $K$ we may find a finite open cover of $C_k$ such that
    $$
    C_k \subset \bigcup_{l = 1}^{M_k} B(y_l, r_l), \quad \mu(\partial B(y_l, r_l)) = 0,\text{ and }\overline{B}(y_l, r_l) \subset B(x_k, r_k) \setminus K.
    $$
    Denote $S_k : = \bigcup_{l = 1}^{M_k} \overline{B}(y_l, r_l)$, and $U := \bigcup_{k = 1}^N (B(x_k, r_k) \setminus S_k)$. Then
    \begin{align*}
        \mu((B(x_k, r_k)\setminus S_k)\setminus K) = \mu((B(x_k, r_k) \setminus K)\setminus S_k) \leq \mu((B(x_k, r_k) \setminus K)\setminus C_k) <\frac{\varepsilon}{N}.
    \end{align*}
    Moreover, $S_k$ and $U$ are continuity sets of $\mu$. That is
    \begin{align*}
        &\mu(\partial S_k) \leq \mu\left(\bigcup_{l=1}^{M_k}\partial B(y_l, r_l)\right) \leq \sum_{l = 1}^{M_k}\mu(\partial B(y_l, r_l)) = 0,\\
        &\mu(\partial U) \leq \mu\left(\bigcup_{k=1}^N \partial(B(x_k, r_k) \setminus S_k)\right) \leq \mu\left(\bigcup_{k=1}^N \partial B(x_k, r_k)\cup \partial S_k\right)\\
        &\qquad\quad\leq \sum_{k=1}^N\mu(\partial B(x_k, r_k)) + \mu(\partial S_k) = 0.
    \end{align*}
    Since $K\subset \R^d \setminus S_k$, then $B(x_k, r_k) \cap K \subset B(x_k, r_k) \setminus S_k$. Hence
    $$
    K = \bigcup_{k=1}^N(B(x_k, r_k)\cap K) \subset \bigcup_{k=1}^N(B(x_k, r_k)\setminus S_k).
    $$
    Furthermore
    $$
    |\mu(U) - \mu(K)| = \mu(U\setminus K) = \mu\left(\bigcup_{k=1}^N (B(x_k, r_k)\setminus S_k)\setminus K\right) \leq \sum_{k=1}^N\mu(B(x_k, r_k)\setminus S_k)\setminus K) < \varepsilon.
    $$
    Finally, by the continuity of $f$, Corollary \ref{cor:portmanteau_for_variation} and Proposition \ref{prop:portmanteau_lemma}
    \begin{align*}
        |\nu|(K)\leq |\nu|(U)\leq \liminf_{n\to +\infty}|\nu_n|(U) \leq \liminf_{n\to +\infty}f(\mu_n(U)) = f(\mu(U)) = |f(\mu(U)) - f(\mu(K))| < \delta, 
    \end{align*}
    for arbitrary $\delta > 0$. Therefore $|\nu|(K) = 0$.
\end{proof}

Following, we define the space of symmetric, matrix-valued measures. The main property of such measures, is that their total variation is controlled by their trace.

\begin{Def}
    We say that $\mu\in \mathcal{M}^+(\R^d; \R^{d\times d}_{\mathrm{sym}})$, iff
    $$
    \int_{\R^d}\phi(x)\xi\otimes\xi:\diff\mu \geq 0,
    $$
    for any $\xi\in \R^d$, $\phi\in C_c(\R^d)$, $\phi\geq 0$.
\end{Def}

\begin{prop}[Lemma 3.2, \cite{woznicki2022weak}]\label{prop:bound_symmetric_measure_trace}
    Suppose that $\mu\in \mathcal{M}^+(\R^d; \R^{d\times d}_{\mathrm{sym}})$, then there exists a constant $C > 0$ such that
    $$
    |\mu|(\R^d) \leq C\,\mathrm{tr}\left(\mu(\R^d)\right).
    $$
\end{prop}

Let us mention that throughout the paper we will use the notation $\mu\otimes\nu$ as a standard product measure (see for example \cite{evans2015measure}), that is
\begin{Def}\label{def:product_measure}
    Suppose $\mu,\nu\in \M(\R^d)$, then for each $S\subset \R^{2d}$
    \begin{align*}
        \mu\otimes\nu (S) := \inf\left\{\sum_{i = 1}^\infty\mu(A_i)\nu(B_i)\right\},
    \end{align*}
    where the infimum is taken over all collections of the $\mu-$measurable sets $A_i\subset\R^d$, and $\nu-$measurable sets $B_i\subset\R^d$ such that
    $$
    S\subset \bigcup_{i = 1}^\infty A_i\times B_i.
    $$
\end{Def}

At last, for the convenience of the reader, we recall the definition of the Wasserstein metric between probability measures. 

\begin{Def}
    Suppose
    $$
    \mu_1, \mu_2\in\left\{\nu\in \mathcal{P}(\R^d): \int_{\R^d}|x|^p\diff\nu(x) < +\infty\right\},
    $$
    then
    $$
    W^p_p(\mu_1,\mu_2) := \min\left\{\int_{\R^d\times\R^d}|x - y|^p\diff\gamma(x, y),\,\,\gamma\in\Pi(\mu_1, \mu_2)\right\},
    $$
    where
    $$
    \Pi(\mu_1. \mu_2) := \left\{\gamma\in \mathcal{P}(\R^d\times\R^d):\,\, (\pi_1)_{\#}\gamma = \mu_1,\,(\pi_2)_{\#}\gamma = \mu_2 \right\}.
    $$
\end{Def}

\section{Approximation and its convergence}\label{Sec:4}
We begin this section by stating a theorem that establishes the existence of solutions to our viscous approximation. As explained in Appendix~\ref{appendix:existence_for_approximation}, this result is a slight modification of the construction in \cite{mucha2025construction}. The proof is long and involves several approximation layers, accompanied by a number of technical lemmas. To maintain readability, we do not reproduce the full argument here. Instead, in Appendix~\ref{appendix:existence_for_approximation}, we present only the modifications required to adapt the approach of \cite{mucha2025construction}.

The rest of this section is then devoted to the proof of Theorem \ref{thm:main_existence}.

\begin{thm}\label{thm:existence_for_approximation}
    Fix $N > 0$. Suppose that the initial conditions $(\rho^N_0, m^N_0)$ satisfy
    $$
        \rho^N_0(x) \geq 0, \quad \sqrt{\rho_0^N}\in W^{1,2}(\R^d),\quad \int_{\R^d}|x|^2\rho_0^N(x)\diff x < +\infty,
    $$
    and for 
    $$
        F(z) = \frac{1 + z^2}{2}\ln(1+ z^2),
    $$
    we have
    $$
        \int_{\R^d}\rho^N_0 F(|u^N_0|)\diff x < +\infty,
    $$
    where $u^N_0 := \frac{m^N_0}{\rho^N_0}$ on the set $\{x\in\R^d\,|\, \rho^N_0(x) > 0\}$, and $D^2K\in C_b(\R^d)$. Then, there exists $(\rho^N, u^N)$, a solution to 
    \begin{align*}
        \left\{\begin{array}{ll}
             \partial_t \rho^n + \DIV_x(\rho^N u^N) = 0, \\
             \partial_t(\rho^N u^N) + \DIV_x(\rho^N u^N\otimes u^N) - \frac{1}{N}\DIV_x(\rho^N\mathbb{D}u^N) + \rho^N u^N(D^2K\star\rho^N) - \rho^N D^2K\star(\rho^N u^N)= 0,
        \end{array}\right.
    \end{align*}
    in a sense that
    \begin{align*}
        \rho^N\in L^\infty(0, T; L^1(\R^d))\cap C([0, T]; L^{3/2}_{loc}(\R^d))&, \quad \sqrt{\rho^N}\in L^\infty(0, T; W^{1,2}(\R^d)),\\
        \sqrt{\rho^N}u^N\in L^\infty(0, T; L^2(\R^d))&,\quad \overline{\sqrt{\rho^N}\mathbb{D}u^N}\in L^2((0, T)\times \R^d),
    \end{align*}
    which satisfy
    \begin{equation}\label{eq:approx_continuity}
        \begin{split}
            \int_0^t\int_{\R^d}\rho^N\partial_t\phi\diff x\diff s + \int_0^t\int_{\R^d}\rho^N u^N\cdot\nabla_x\phi\diff x\diff s = \int_{\R^d}\rho^N(t, x)\phi(t, x)\diff x - \int_{\R^d}\rho_0^N(x)\phi(0, x)\diff x,
        \end{split}
    \end{equation}
    for all $t\in [0, T]$ and $\phi \in C^1_c([0, T]\times \R^d)$, as well as
    \begin{equation}\label{eq:approx_momentum}
        \begin{split}
            &\int_0^t\int_{\R^d}\rho^N u^N\cdot\partial_t\phi + \rho^N u^N\otimes u^N : \nabla_x\phi - \frac{1}{N}\sqrt{\rho^N}\overline{\sqrt{\rho^N}\mathbb{D}u^n} : \nabla_x\phi \diff x\diff s\\
            &\qquad- \int_0^t\int_{\R^d}\rho^N(x)\phi\int_{\R^d}D^2K(x-y)(u^N(x) - u^N(y))\rho^N(y)\diff y\diff x\diff s\\
            &= \int_{\R^d}\rho^N(t, x)u^N(t, x)\cdot\phi^N(t, x)\diff x - \int_{\R^d}\rho_0^N(x) u^N_0(x)\cdot \phi(0, x)\diff x,
        \end{split}
    \end{equation}
    for a.e. $t\in (0, T)$ and $\phi\in C^1_c([0, T)\times \R^d; \R^d)$. Moreover, for any $\phi\in C^2_c([0, T)\times \R^d; \R^d)$
    \begin{equation*}
        \begin{split}
            &\int_0^t\int_{\R^d}\sqrt{\rho^N}\overline{\sqrt{\rho^N}\mathbb{D}u^n} : \nabla_x\phi \diff x\diff s\\
            &= \int_0^t\int_{\R^d}\rho^N u^N(\Delta\phi + \nabla_x\DIV_x\phi) + 2\left(\nabla_x\sqrt{\rho^N}\otimes \sqrt{\rho^N}u^N\right):\nabla_x\phi\diff x\diff s,
        \end{split}
    \end{equation*}
    and the following relations are satisfied
    \begin{enumerate}
        \item mass conservation
    \begin{equation}\label{eq:approx_mass_conservation}
            \int_{\R^d}\rho^N(t, x)\diff x = \int_{\R^d}\rho^N_0(x)\diff x,
    \end{equation}
    for all $t\in [0, T]$.
    \item energy estimates
    \begin{equation}\label{ineq:approx_energy_ineq}
        \begin{split}
            &\int_{\R^d}\rho^N(t, x)|u^N(t, x)|^2\diff x + \frac{2}{N}\int_0^t\int_{\R^d}\left|\overline{\sqrt{\rho^N}\mathbb{D}u^N}\right|^2\diff x\diff s \\
            &\qquad\qquad \leq 4t\|D^2K\|_\infty e^{4t\|D^2K\|_\infty}\int_{\R^d}\rho^N_0(x)|u^N_0(x)|^2\diff x,\\
            &\int_{\R^d}\rho^N(t, x)|u^N(t, x)|^2\diff x\leq  e^{4t\|D^2K\|_\infty}\int_{\R^d}\rho^N_0(x)|u^N_0(x)|^2\diff x.
        \end{split}
    \end{equation}
    for a.e. $t\in (0, T)$. Moreover, if $D^2K$ is even and positive semi-definite, then
    \begin{equation}\label{ineq:approx_energy_ineq_symmetric}
        \begin{split}
            &\int_{\R^d}\rho^N(t, x)|u^N(t, x)|^2\diff x  + \frac{2}{N}\int_0^t\int_{\R^d}\left|\overline{\sqrt{\rho^N}\mathbb{D}u^N}\right|^2\diff x\diff s \\
            & + \int_0^t\int_{\R^{d\times d}}\rho^N(s, x)\rho^N(s, y)(u^N(s, x) - u^N(s, y))D^2K(x - y)(u^N(s, x) - u^N(s, y))\diff x\diff y\diff s\\
            &\qquad\qquad \leq \int_{\R^d}\rho^N_0(x)|u^N_0(x)|^2\diff x,
        \end{split}
    \end{equation}
    for a.e. $t\in (0, T)$.
    \item second moment estimate
    \begin{equation}\label{ineq:approx_second_moment_ineq}
        \begin{split}
            \int_{\R^d}|x|^2\rho^N(t, x)\diff x \leq C(T, \|D^2K\|_\infty, \|\rho_0^N|u_0^N|^2\|_{L^1_x})\int_{\R^d}|x|^2\rho^N_0(x)\diff x,
        \end{split}
    \end{equation}
    for a.e. $t\in (0, T)$.
    \end{enumerate}
\end{thm}
\begin{proof}
    We refer the reader to Appendix \ref{appendix:existence_for_approximation}.
\end{proof}

Before moving forward, we  show another set of a priori bounds for the solutions $(\rho^N, u^N)$ given by Theorem \ref{thm:existence_for_approximation}.

\begin{lem}\label{lem:bounds_for_lipshitz_distance}
    Let $(\rho^N, u^N)$ be given by Theorem \ref{thm:existence_for_approximation}, then
    \begin{enumerate}
        \item For every $\varphi\in C_c^\infty(\R^d)$, $\left\{\frac{d}{dt}\int_{\R^d}\varphi(x)\rho^N(t, x)\diff x\right\}_{N\in\N}$ is bounded in $L^\infty(0, T)$ with a constant depending only on initial data and $\|\varphi\|_\infty$, $\|\nabla_x \varphi\|_{\infty}$,\label{bound:approx_Lipsh_density}
        \item For every $\varphi\in C_c^\infty(\R^d)$, $\left\{\frac{d}{dt}\int_{\R^d}\varphi(x)\cdot \rho^N u^N\diff x\right\}_{N\in\N}$ is bounded in $L^2(0, T)$ with a constant depending only on initial data and $\|\varphi\|_\infty$, $\|\nabla_x \varphi\|_{\infty}$.\label{bound:approx_Lipsh_momentum}
    \end{enumerate}
\end{lem}

\begin{proof}
    To prove the first claim we test \eqref{eq:approx_continuity} by $\varphi\in C^\infty_c(\R^d)$. We get
    $$
    \int_{\R^d}\varphi(x)\rho^N(t, x)\diff x - \int_{\R^d}\varphi(x)\rho^N_0(x)\diff x = \int_0^t\int_{\R^d}\rho^N u^N\cdot\nabla_x\varphi\diff x\diff s,
    $$
    hence
    $$
    \frac{d}{dt}\int_{\R^d}\varphi(x)\rho^N(t, x)\diff x = \int_{\R^d}\rho^N u^N\cdot\nabla_x\varphi\diff x,\text{ for a.e. }t\in (0, T).
    $$
    Thus, by H\"{o}lder's inequality we obtain
    $$
    \left|\frac{d}{dt}\int_{\R^d}\varphi(x)\rho^N(t, x)\diff x\right| \leq \|\nabla_x\varphi\|_\infty\int_{\R^d}|\rho^N u^N|\diff x \leq \|\nabla_x\varphi\|_\infty \|\rho^N\|_{L^\infty_t L^1_x}^{1/2}\|\rho^N|u^N|^2\|^{1/2}_{L^\infty_t L^1_x},
    $$
    which is bounded by \eqref{eq:approx_mass_conservation} and \eqref{ineq:approx_energy_ineq}. 
    
    For the second claim, we want to test the momentum equation \eqref{eq:approx_momentum} by $\varphi\in C_c^\infty(\R^d)$. Since in this case we need a function with a compact support in $[0, T)$, we may define
    \begin{align*}
        \phi(t, x) := \psi(t)\varphi(x),
    \end{align*}
    where $\psi\equiv 1$ on $[0, t]$, and it smoothly goes to $0$ on an interval $[t, t+\delta) \subset [t, T)$. Then, as $t\in (0, T)$ is arbitrary  
    \begin{align*}
        &\left|\frac{d}{dt}\int_{\R^d}\varphi(x)\cdot\rho^N(t, x)u^N(t, x)\diff x\right| \leq \|\nabla_x\varphi\|_\infty\left(\|\rho^N|u^N|^2\|_{L^\infty_t L^1_x} + \frac{1}{N}\|\rho^N\|_{L^\infty_t L^1_x}\left\|\overline{\sqrt{\rho^N}\mathbb{D}u^N}\right\|_{L^2_x}\right)\\
        &\qquad + \|\varphi\|_\infty\|D^2K\|_\infty\int_{\R^d}\rho^N|u^N|\diff x\int_{\R^d}\rho^N\diff y + \|\varphi\|_\infty\|D^2K\|_\infty\int_{\R^d}\rho^N\diff x\int_{\R^d}\rho^N|u^N|\diff y\\
        &\quad\leq \|\nabla_x\varphi\|_\infty\left(\|\rho^N|u^N|^2\|_{L^\infty_t L^1_x} + \frac{1}{N}\|\rho^N\|_{L^\infty_t L^1_x}\left\|\overline{\sqrt{\rho^N}\mathbb{D}u^N}\right\|_{L^2_x}\right)\\
        &\qquad\quad + 2\|\varphi\|_\infty\|D^2K\|_\infty\|\rho^N\|_{L^\infty_t L^1_x}^{3/2}\|\rho^N|u^N|^2\|_{L^\infty_t L^1_x}^{1/2},
    \end{align*}
    for a.e. $t\in (0, T)$. Here, the right-hand side is bounded in $L^2(0, T)$ by \eqref{eq:approx_mass_conservation} and \eqref{ineq:approx_energy_ineq}.
\end{proof}

Having the existence theorem for approximation, and the sufficient a priori estimates, we can move to the existence proof.

\begin{proof}[Proof of Theorem \ref{thm:main_existence} (Part $1$)]

To start off the proof, let us take a sequence of the approximation of our initial data. Let $\rho_0^N\in C^\infty_c(\R^d)$ be such that
$\rho_0^N\diff x\in \mathcal{P}(\R^d)$ and
$$
\rho^N_0\diff x \rightharpoonup \rho_0,\text{ weakly in }\mathcal{M}^+(\R^d).
$$
Note that such a sequence can be found by a standard considerations involving mollifications of the measure $\rho_0$ and multiplication by the smooth cut-off functions, which ensure compactness of the support. Moreover, since
$$
\int_{\R^d}|x|^2\diff \rho_0 < +\infty,
$$
the sequence $\{\rho^N_0\}_{N\in \N}$ can be taken, so that it satisfies
$$
\sup_{N\in\N}\int_{\R^d}|x|^2\rho^N_0\diff x < +\infty.
$$
For the velocity, we simply fix $u_0^N = u_0\in C_b(\R^d)$. Such a sequence $\{(\rho^N_0, u_0^N)\}_{N\in \N}$ satisfies the assumptions in Theorem \ref{thm:existence_for_approximation}. With this set-up, we may prove the following convergence lemma.

\begin{lem}\label{lem:measure_convergence_lemma}
    Let $(\rho^N, u^N)$ be as in Theorem \ref{thm:existence_for_approximation}, with the initial data $(\rho_0^N, u_0^N)$ given above. Then, there exist $\rho\in C([0, T]; (\M^+(\R^d), d_f))$ and $m\in L^\infty(0, T; (\M(\R^d), d_f))$, such that (up to the subsequence which we do not relabel)
    \begin{align}
        \rho^N &\rightarrow \rho &&\text{ in }C([0, T]; (\M^+(\R^d), d_f)),\label{N_conv_rho}\\
        \rho^N u^N =: m^N &\rightarrow m &&\text{ in }L^\infty(0, T; (\M(\R^d), d_f)).\label{N_conv_m}
    \end{align}
    Moreover, for almost every $t\in (0, T)$, $m_t << \rho_t$, and there exists $u\in L^\infty(0, T; L^2(\R^d, \diff \rho_t))$ such that
    $$
    \diff m_t = u(t, \cdot)\diff \rho_t.
    $$
    Furthermore,
    \begin{align}\label{N_mass_conserv}
        \int_{\R^d}\diff\rho_t(x) = \int_{\R^d}\diff\rho_0(x) = 1,
    \end{align}
    and the second moment and the first moment of respectively $\rho_t$ and $m_t$ are bounded, that is
    \begin{align}
        &\sup_{t\in (0, T)}\int_{\R^d}|x|^2\diff\rho_t < + \infty, \label{bound:second_moment_density}\\
        &\sup_{t\in (0, T)}\int_{\R^d}|x|\diff|m_t| < +\infty.\label{bound:first_moment_momentum}
    \end{align}
\end{lem}
\begin{proof}
    First, we show the convergence for $\{\rho^N\}_{N\in\N}$. We aim to use the Arzel\`{a}-- Ascoli theorem. The boundedness of the sequence in $C([0, T]; (\M^+(\R^d), d_f))$ is a direct consequence of \eqref{eq:approx_mass_conservation}. Indeed,
    \begin{align*}
    \sup_{N\in \N}\sup_{t\in [0, T]}d_f(\rho^N(t), 0) \leq \sup_{N\in \N}\sup_{t\in [0, T]}\int_{\R^d}\rho^N(t)\diff x = 1.
    \end{align*}
    Moreover, by the second moment bound in \eqref{ineq:approx_second_moment_ineq}, and the Chebyshev inequality we deduce the tightness of the family $\{\rho^N(t)\}_{N\in\N}$. Hence, by Proposition \ref{prop:tight_sequence_equivalences} we see, that our sequence is relatively compact. At last we need to check the equi-continuity in flat metric. Fix an arbitrary $\varphi\in C_c^\infty(\R^d)$, and $t, s\in [0, T]$. Then, by  Lemma \ref{lem:bounds_for_lipshitz_distance}
    \begin{align}\label{ineq:arzela_ascoli_argument_modulus_of_cont}
        \left|\int_{\R^d}\varphi(\rho^N(t) - \rho^N(s))\diff x\right| = \left|\int_s^t\frac{d}{dt}\int_{\R^d}\varphi\,\rho^N\diff x\diff\tau\right| \leq C(\|\varphi\|_\infty, \|\nabla_x\varphi\|_\infty)|t - s|.
    \end{align}
    Since the constant above depends only on $\|\varphi\|_\infty$ and $\|\nabla_x \varphi\|_\infty$ we can take supremum over $\varphi$ in \eqref{ineq:arzela_ascoli_argument_modulus_of_cont} to obtain
    $$
    d_f(\rho^N(t), \rho^N(s)) \leq C|t-s|,
    $$
    which is the last ingredient to use the Arzel\`{a}--Ascoli theorem and conclude the existence of $\rho\in C([0, T]; (\M^+(\R^d), d_f))$, such that
    $$
    \rho^N \rightarrow \rho \text{ in }C([0, T]; (\M^+(\R^d), d_f)),
    $$
    which finishes the proof of \eqref{N_conv_rho}.
    
    When it comes to the momentum, we can again utilize Lemma \ref{lem:bounds_for_lipshitz_distance} to obtain
    \begin{align}\label{ineq:arzela_ascoli_argument_modulus_of_cont_momentum}
        \left|\int_{\R^d}\varphi(\rho^N(t)u^N(t) - \rho^N(s)u^N(s))\diff x\right| = \left|\int_s^t\frac{d}{dt}\int_{\R^d}\varphi\,\rho^Nu^N\diff x\diff\tau\right| \leq C(\|\varphi\|_\infty, \|\nabla_x\varphi\|_\infty)\sqrt{|t - s|},
    \end{align}
    for a.e. $t, s\in (0, T)$. Since this holds only almost everywhere, it is not enough to proceed in a space of continuous functions, but only in $L^\infty$. Here, we may follow an argument from \cite{cherkas1970compactness}, but as it is short, and only given for functions taking real or complex values, and not in an arbitrary complete metric space, we will give the whole argument with appropriate changes for the convenience of the reader. 
    
    By \eqref{ineq:approx_energy_ineq}, \eqref{eq:approx_mass_conservation}, and the H\"{o}lder inequality
    \begin{align*}
        \sup_{N\in \N}\sup_{t\in (0, T)}\int_{\R^d}|m^N|\diff x \leq \sup_{N\in \N}\sup_{t\in (0, T)}\sqrt{\int_{\R^d}\rho^N(t)\diff x}\sqrt{\int_{\R^d}\rho^N(t)|u^N(t)|^2\diff x} \leq C.
    \end{align*}
    Moreover using \eqref{ineq:approx_energy_ineq}, \eqref{ineq:approx_second_moment_ineq}, and Young's inequality
    \begin{align*}
        \sup_{t\in (0, T)}\int_{\R^d}|x||m^N|\diff x \leq \frac{1}{2}\sup_{t\in (0, T)}\int_{\R^d}|x|^2\rho^N\diff x + \frac{1}{2}\sup_{t\in (0, T)}\int_{\R^d}\rho^N|u^N|^2\diff x \leq C.
    \end{align*}
    Let $\mathbf{m}^N$ be a particular function chosen from the equivalence class $m^N$, such that $\|\mathbf{m}^N(t)\|_{TV} \leq C$, and $\||x|\mathbf{m}^N(t)\|_{TV} \leq C$ for every $t\in [0, T]$. For any non-empty, measurable set $E\subset [0, T]$ define
    $$
    |\mathbf{m}^N|_E := \inf_{R}\sup_{t,s\in E\setminus R}d_f(\mathbf{m}^N(t), \mathbf{m}^N(s)),
    $$
    where the infimum is taken over all sets $R\subset[0, T]$ of measure zero. By virtue of \eqref{ineq:arzela_ascoli_argument_modulus_of_cont_momentum} there exists a partition $\{E_{k, j}\}_{j=1}^{l_k}$, such that
    \begin{align*}
        \max_j\sup_{N}|\mathbf{m}^N|_{E_{k, j}} \leq \frac{1}{k}.
    \end{align*}
    Moreover, for each $E_{k, j}$, we may find a set $F_{k, j}\subseteq E_{k, j}$, such that $E_{k, j}\setminus F_{k, j}$ is of measure zero, and
    \begin{align*}
        \sup_{N}\sup_{t, s\in F_{k, j}}d_f(\mathbf{m}^N(t), \mathbf{m}^N(s)) \leq \frac{1}{k}.
    \end{align*}
    Now, from any $F_{k, j}$ we select an arbitrary point $s_{k, j}$. Then, the sequence $\{\mathbf{m}^N(s_{k, j})\}$ is bounded in $(\mathcal{M}(\R^d), \|\cdot\|_{TV})$, and since its first moment is bounded, by the Chebyshev inequality, it is also tight. Hence, for $k = 1$, by Proposition \ref{prop:tight_sequence_equivalences}, there exists a subsequence, which we do not relabel, such that $\{\mathbf{m}^N\}_{N\in \N}$ converges in $(\mathcal{M}(\R^d), d_f)$ in every point $\{s_{1, j}\}_{j = 1}^{l_1}$. Since the number of points is finite, we can make this convence uniform with respect to them. Using a diagonal argument, we can find a subsequence (which we again do not relabel), such that $\{\mathbf{m}^N\}_{N\in\N}$ converges in every point $\bigcup_{k\in \N}\{s_{k, j}\}_{j = 1}^{l_k}$, and this convergence is uniform on finite subsets. Fix $\varepsilon > 0$, and $k\in \N$ for which $\frac{1}{k} < \frac{\varepsilon}{3}$. Since $\{\mathbf{m}^N\}$ converges uniformly on $\{s_{k, j}\}_{j = 1}^{l_k}$, then there exists some $N_\varepsilon\in \N$ such that if $N, M\geq N_{\varepsilon}$
    \begin{align*}
        \max_{j}d_f(\mathbf{m}^N(s_{k, j}), \mathbf{m}^M(s_{k, j})) \leq \frac{\varepsilon}{3}.
    \end{align*}
    For any $s\in F_{k, j}$ we have
    \begin{align*}
        d_f(\mathbf{m}^N(s), \mathbf{m}^M(s)) \leq d_f(\mathbf{m}^N(s), \mathbf{m}^N(s_{k, j})) + d_f(\mathbf{m}^N(s_{k, j}), \mathbf{m}^M(s_{k, j})) + d_f(\mathbf{m}^M(s_{k, j}), \mathbf{m}^M(s)) < \varepsilon,
    \end{align*}
    thus 
    $$
    \|m^N - m^M\|_{L^\infty(0, T; (\mathcal{M}(\R^d), d_f))} < \varepsilon.
    $$
    It is easy to see that (by Proposition \ref{prop:tight_sequence_equivalences}, see also the argument for \eqref{bound_second_moment_in_approx}) $\mathcal{K} := \{\mu\in \mathcal{M}(\R^d):\, |\mu|(\R^d)\leq C,\, \int_{\R^d}|x|\diff|\mu|(x) \leq C\}$ is complete with respect to the $d_f$ metric, thus so is $L^\infty(0, T; (\mathcal{K}, d_f))$. Hence, the Cauchy sequence $\{m^N\}$ has a limit $m\in L^\infty(0, T; (\mathcal{K}, d_f))$, and thus \eqref{N_conv_m} is verified.
    
    Moving forward, we want to differentiate $m$ with respect to $\rho$. Using H\"{o}lder inequality and \eqref{ineq:approx_energy_ineq}, we can easily see that for any Borel $A\subset\R^d$
    \begin{align*}
        |m_t^N|(A) = \int_A \rho^N\,|u^N|\diff x \leq C\sqrt{\int_A \rho^N\diff x} = C\sqrt{\rho^N_t(A)}.
    \end{align*}
    Thus, by Lemma \ref{lem:differentiation_of_limit_measure} $m_t << \rho_t$. Hence, also $\diff m_t\otimes \diff t << \diff\rho_t \otimes \diff t$ (see Definition \ref{def:product_measure} for the definition of the product measure), and by the Radon--Nikodym theorem, there exists a function $u\in L^1((0, T)\times \R^d)$ such that
    $$
    \diff m_t\otimes\diff t = u(t, x)\diff \rho_t\otimes\diff t,
    $$
    which readily implies that
    $$
        \diff m_t = u(t,\cdot)\diff\rho_t \text{ for almost every }t\in (0, T).
    $$
    Thus, using Proposition \ref{prop:lowersemicont_measures} and \eqref{ineq:approx_energy_ineq}
    \begin{align*}
        \int_{\R^d}|u(t, x)|^2\diff\rho_t(x) \leq e^{4t\|D^2K\|_{\infty}}\int_{\R^d}|u_0|^2\diff\rho_0(x).
    \end{align*}
    In particular
    $$
        u\in L^\infty(0, T; L^2(\diff\rho_t)).
    $$
    To show the bounds on the moments, we utilize the Fubini theorem, Proposition \ref{prop:portmanteau_lemma}, the Fatou lemma, and \eqref{ineq:approx_second_moment_ineq}
    \begin{equation}\label{bound_second_moment_in_approx}
        \begin{split}
            \int_{\R^d}|x|^2\diff\rho_t &= \int_0^\infty \rho_t\{|x|^2 > s\}\diff s \leq \int_0^\infty \liminf_{N\to +\infty}\rho^N_t\{|x|^2 > s\}\diff s\\
            &\leq \liminf_{N\to + \infty}\int_0^\infty\rho^N_t\{|x|^2 > s\}\diff s = \liminf_{N\to +\infty}\int_{\R^d}|x|^2\rho^N\diff x < +\infty.
        \end{split}
    \end{equation}
    At last, to show the mass conservation, we note that $\R^d$ is trivially a continuity set of the measure $\rho_t$, hence, by Proposition \ref{prop:portmanteau_lemma}
    $$
    1 = \lim_{N\to +\infty}\int_{\R^d}\rho^N(t, x)\diff x = \int_{\R^d}\diff\rho_t(x).
    $$
    The proof of Lemma \ref{lem:measure_convergence_lemma} is thus complete.
\end{proof}

Let us now return to the proof of Theorem \ref{thm:main_existence} (Part $1$).
The regularity and the moment bounds of $\rho$ and $m$ necessary for Definition \ref{def:measure_solution} are given already in Lemma \ref{lem:measure_convergence_lemma}. The aforementioned lemma allows also to converge in the   approximate continuity equation \eqref{eq:approx_continuity} to \eqref{eq:final_mass_conservation}. To converge in \eqref{eq:approx_momentum}, the terms containing only $\rho^N u^N$ are easy to handle. For a convective term we use the embedding
$$
L^\infty(0, T; L^1(\R^d)) \hookrightarrow L^\infty(0, T; \mathcal{M}(\R^d)),
$$
which, using the Banach--Alaouglu theorem, allows us to define
$$
\mu := \mathrm{weak^*}\lim_{N\to \infty}(\rho^Nu^N\otimes u^N) -  u\otimes u\diff\rho.
$$
To show that $\mu_t\in \M^+(\R^d; \R^{d\times d}_{\mathrm{sym}})$, we note that by Proposition \ref{prop:lowersemicont_measures}, for any $\xi\in \R^d$, $\phi\in C_c(\R^d)$, $\phi \geq 0$
\begin{align*}
    &\int_{\R^d}\phi(x)\xi\otimes\xi : \diff\mu_t = \lim_{N\to+\infty}\int_{\R^d}\phi(x)\rho^N\xi\otimes\xi : u^N\otimes u^N\diff x - \int_{\R^d}\phi(x)\xi\otimes\xi:u\otimes u\diff\rho_t\\
    &\qquad\geq\lim_{N\to+\infty}\int_{\R^d}\phi(x)\rho^N|\xi u^N|^2\diff x - \int_{\R^d}\phi(x)|\xi u|^2\diff\rho_t \geq 0.
\end{align*}
Concerning the term with the symmetric gradient of $u^N$, we note that it disappears  by \eqref{eq:approx_continuity} and \eqref{ineq:approx_energy_ineq}, as we have by H\"{o}lder's inequality
\begin{equation*}
    \begin{split}
        \left|\int_0^t\int_{\R^d}\frac{1}{N}\sqrt{\rho^N}\overline{\sqrt{\rho^N}\mathbb{D}u^n} : \nabla_x\phi \diff x\diff s\right| \leq \frac{1}{\sqrt{N}}\|\nabla_x\phi\|_\infty\|\rho^N\|_{L^1_{t,x}}^{1/2}\left\|\frac{1}{\sqrt{N}}\overline{\sqrt{\rho^N}\mathbb{D}u^n}\right\|_{L^2_{t,x}} \rightarrow 0,
    \end{split}
\end{equation*}
as $N\to +\infty$. 

At last, we look at the convergence of the non-local term. Let $\phi\in C^1_c([0, T)\times \R^d)$ be arbitrary. By \eqref{N_mass_conserv} and \eqref{bound:first_moment_momentum} we get tightness of the sequence $\{\diff m^N_t\otimes \diff\rho^N_t\}$. Thus there exist compact sets $V_1, V_2\subset \R^{d}$ (note, that these sets do not depend on $t$ as the bounds on the second moment of $\rho^N_t$ and the first moment of $m^N_t$ are also independent of $t$) such that for every $N\in \N$
\begin{align}\label{thightness_of_product_sequence}
|m^N_t\otimes \rho^N_t|(\R^{2d}\setminus (V_1\times V_2)) < \varepsilon.
\end{align}
Since the measure $\diff m_t\otimes \diff\rho_t$ is tight as well, we can take $V_1\times V_2$  large enough, so that
\begin{align}\label{tightness_of_product}
|m_t\otimes \rho_t|(\R^{2d}\setminus(V_1\times V_2)) < \varepsilon.
\end{align}
Because function
$$
\Phi(x, y) := D^2K(x-y)
$$
is in $C(V_1\times V_2)$, by the Stone--Weierstrass theorem, we may find functions $\psi_i\in C(V_i)$ such that
\begin{align}\label{stone_weierstrass_uniform}
\|\Phi(x,y) - \psi_1(x)\psi_2(y)\|_{C(V_1\times V_2)} < \varepsilon.
\end{align}
Then,
\begin{equation*}
    \begin{split}
        &\left|\int_0^t\int_{\R^d}\phi\,\rho^N(x)\int_{\R^d}D^2K(x-y)u^N(x)\rho^N(y)\diff y\diff x\diff\tau - \int_0^t\int_{\R^d}\phi\int_{\R^d}D^2K(x-y)\diff\rho_t(y)\diff m_t(x)\diff\tau\right| \\
        &\leq \left|\int_0^t\int_{\R^{2d}\setminus(V_1\times V_2)}\phi\,\rho^N(x)D^2K(x-y)u^N(x)\rho^N(y)\diff y\diff x\diff\tau\right|\\
        &\qquad +\left|\int_0^t\int_{\R^{2d}\setminus(V_1\times V_2)}\phi\,D^2K(x-y)\diff\rho_t(y)\diff m_t(x)\diff\tau\right|\\
        &\qquad + \Bigg|\int_0^t\int_{V_1\times V_2}\phi\,\rho^N(x) D^2K(x-y)u^N(x)\rho^N(y)\diff y\diff x\diff\tau\\
        &\qquad\qquad- \int_0^t\int_{V_1\times V_2}\phi\,\rho^N(x) \psi_1(x)\psi_2(y)u^N(x)\rho^N(y)\diff y\diff x\diff\tau\Bigg|\\
        &\qquad + \left|\int_0^t\int_{V_1\times V_2}\phi\,\psi_1(x)\psi_2(y)\diff\rho_t(y)\diff m_t(x)\diff\tau- \int_0^t\int_{V_1\times V_2}\phi\,D^2K(x-y)\diff\rho_t(y)\diff m_t(x)\diff\tau\right|\\
        &\qquad + \left|\int_0^t\int_{V_1\times V_2}\phi\,\rho^N(x) \psi_1(x)\psi_2(y)u^N(x)\rho^N(y)\diff y\diff x\diff\tau - \int_0^t\int_{V_1\times V_2}\phi\,\psi_1(x)\psi_2(y)\diff\rho_t(y)\diff m_t(x)\diff\tau\right|\\
        & = I_1 + I_2 + I_3 + I_4 + I_5.
    \end{split}
\end{equation*}
By \eqref{thightness_of_product_sequence} and \eqref{tightness_of_product}
$$
I_1 + I_2 \leq 2\|\phi\|_\infty\|D^2K\|_\infty\varepsilon,
$$
and by \eqref{stone_weierstrass_uniform}, together with \eqref{ineq:approx_energy_ineq}
$$
I_3 + I_4 \leq C\|\phi\|_\infty \varepsilon.
$$
Finally, for the last term, using the convergences \eqref{N_conv_rho}, \eqref{N_conv_m} from Lemma \ref{lem:measure_convergence_lemma}, using product structure of the integral, we get 
$$
I_5 \to 0, \text{ as }N\to +\infty.
$$
Since $\varepsilon>0$ is arbitrary we get the needed convergence. With this, we have proven the weak formulations \eqref{eq:final_mass_conservation} and \eqref{eq:final_momentum_conservation}, whenever $\phi\in C^1_c([0, T]\times \R^d)$ or $\phi\in C^1_c([0, T)\times \R^d)$ respectively. We want to extend the class of admissible test functions. As the extension for \eqref{eq:final_momentum_conservation} is similar as for \eqref{eq:final_mass_conservation}, we will only explain it for \eqref{eq:final_mass_conservation}. Fix $\phi\in C^1([0, T]\times \R^d)$, such that $|\phi(t, x)|, |\partial_t\phi(t, x)|\leq C(1 + |x|^2)$, $|\nabla_x\phi(t, x)|\leq C(1 + |x|)$, for some $C > 0$. Let $\psi_R\in C^\infty_c(\R^d)$ be a standard cut-off function, that is $\psi_R\equiv 1$ on $B(0, R)$, $\psi_R\equiv 0$ on $\R^d \backslash B(0, R+1)$, $0 \leq \psi_R \leq 1$ and $|\nabla_x \psi_R| \leq 1$. Then, the function $\Psi_R(t, x) := \psi_R(x)\phi(t, x)$ is in $C^1_c([0, T]\times \R^d)$, which means it is an admissible test function for \eqref{eq:final_mass_conservation}. With this set up we have
\begin{align*}
    \Psi_R (t, x) &\rightarrow \phi(t, x),\text{ pointwisely in } [0, T]\times \R^d,\\
    \partial_t\Psi_R (t, x) &\rightarrow \partial_t\phi(t, x),\text{ pointwisely in } [0, T]\times \R^d,\\
    \nabla_x\Psi_R (t, x) &\rightarrow \nabla_x\phi(t, x),\text{ pointwisely in } [0, T]\times \R^d.
\end{align*}
Moreover 
\begin{align*}
    |\Psi_R(t, x)| &\leq C(1 + |x|^2) \in L^1(0, T;L^1(\diff\rho_t)),\\
    |\partial_t\Psi_R(t, x)| &\leq C(1 + |x|^2)\in L^1(0, T; L^1(\diff\rho_t)),\\
    |\nabla_x\Psi_R(t, x)| &\leq C(1 + |x|)\in L^1(0, T; L^1(\diff m_t)),
\end{align*}
by \eqref{bound:second_moment_density} and \eqref{bound:first_moment_momentum}. Hence, by Lebesgue's dominated convergence theorem we can converge with $R\to +\infty$ in \eqref{eq:final_mass_conservation} tested by $\Psi_R$, and obtain the extension of possible test functions.

Now, we focus our attention on the energy inequality \eqref{ineq:final_energy_inequality}. Note that due to the fact that trace is a linear function, it is invariant with respect to the weak and weak* convergence. Thus, by the virtue of \eqref{ineq:approx_energy_ineq}
\begin{equation}\label{ineq:proof_N_conv_energy}
\begin{split}
    &\frac{1}{2}\int_{\R^d}|u(t, x)|^2\diff\rho_t(x) + \frac{1}{2}\mathrm{tr}(\mu_t(\R^d))\\
    &\phantom{=}\leq \frac{1}{2}\int_{\R^d}|u(t, x)|^2\diff\rho_t(x) + \frac{1}{2}\liminf_{N\to +\infty}\int_{\R^d}\mathrm{tr}(\rho^N u^N\otimes u^N)\diff x -  \frac{1}{2}\int_{\R^d}\mathrm{tr}(u\otimes u)\diff \rho_t(x)\\
    &\phantom{=}= \frac{1}{2}\liminf_{N\to +\infty}\int_{\R^d}\rho^N |u^N|^2\diff x \leq e^{4t\|D^2K\|_\infty}\frac{1}{2}\int_{\R^d}|u_0|^2\diff\rho_0,
\end{split}
\end{equation}
with a similar approach, whenever $D^2K$ is even and positive semi-definite.
\end{proof}

Here, we have shown the first part of  Theorem \ref{thm:main_existence}, under the assumption that the initial velocity is regular, i.e.  $u_0\in C_b(\R^d)$. We will get rid of this restriction in the third part of the proof. For now, we  move on to proving the second part of  Theorem \ref{thm:main_existence} for $u_0\in C_b(\R^d)$, with $w$ defined in \eqref{eq:final_compatibility_rho_w}. 

As the entry to our discussion, we note that  $D^2 K\in C_b(\R^d)$ implies that $\nabla_x K$ is Lipschitz, hence
\begin{align}\label{ineq:lipshitz_nabla_K}
    |\nabla_x K(x)| \leq C(\|D^2K\|_\infty)(1 + |x|),
\end{align}
which will be heavily exploited in the arguments below. 

In what follows we state a simple reformulation of Theorem \ref{thm:existence_for_approximation}.

\begin{thm}\label{thm:existence_for_approximation_rho_w}
    Fix $N > 0$. Suppose that the initial conditions $(\rho^N_0, m^N_0)$ satisfy the assumptions of the Theorem \ref{thm:existence_for_approximation}, and that $(\rho^N, u^N)$ is the given solution. Define
    \begin{align}\label{eq:def_of_w_N}
    w^N := u^N + \nabla_x K\star\rho^N.
    \end{align}
    Then, $(\rho^N, w^N, u^N)$ is a solution to 
    \begin{align*}
        \left\{\begin{array}{ll}
             \partial_t \rho^n + \DIV_x(\rho^N u^N) = 0, \\
             \partial_t(\rho^N w^N) + \DIV_x(\rho^N w^N\otimes u^N) - \frac{1}{N}\DIV_x(\rho^N\mathbb{D}u^N) = 0,
        \end{array}\right.
    \end{align*}
    in a sense that
    \begin{align*}
        \rho^N\in L^\infty(0, T; L^1(\R^d)), \quad \sqrt{\rho^N}\in L^\infty&(0, T; W^{1,2}(\R^d)),\quad \sqrt{\rho^N}w^N\in L^\infty(0, T; L^2(\R^d))\\
        \sqrt{\rho^N}u^N\in L^\infty(0, T; L^2(\R^d))&,\quad \overline{\sqrt{\rho^N}\mathbb{D}u^N}\in L^2((0, T)\times \R^d).
    \end{align*}
   The triple $(\rho^N, w^N, u^N)$ satisfies
    \begin{equation}\label{eq:approx_rho_w_continuity}
        \begin{split}
            \int_0^t\int_{\R^d}\rho^N\partial_t\phi\diff x\diff s + \int_0^t\int_{\R^d}\rho^N u^N\cdot\nabla_x\phi\diff x\diff s = \int_{\R^d}\rho^N(t, x)\phi(t, x)\diff x - \int_{\R^d}\rho_0^N(x)\phi(0, x)\diff x,
        \end{split}
    \end{equation}
    for a.e. $t\in (0, T)$ and $\phi \in C^1_c([0, T)\times \R^d)$, as well as
    \begin{equation}\label{eq:approx_rho_w_momentum}
        \begin{split}
            &\int_0^t\int_{\R^d}\rho^N w^N\cdot\partial_t\phi + \rho^N w^N\otimes u^N : \nabla_x\phi - \frac{1}{N}\sqrt{\rho^N}\overline{\sqrt{\rho^N}\mathbb{D}u^n} : \nabla_x\phi \diff x\diff s\\
            &= \int_{\R^d}\rho^N(t, x)w^N(t, x)\cdot\phi(t, x)\diff x - \int_{\R^d}\rho_0^N(x) w^N_0(x)\cdot \phi(0, x)\diff x,
        \end{split}
    \end{equation}
    for a.e. $t\in (0, T)$ and $\phi\in C^1_c([0, T)\times \R^d)$. Moreover, for any $\phi\in C^2_c([0, T)\times \R^d)$ 
    \begin{equation*}
        \begin{split}
            &\int_0^t\int_{\R^d}\sqrt{\rho^N}\overline{\sqrt{\rho^N}\mathbb{D}u^n} : \nabla_x\phi \diff x\diff s\\
            &= \int_0^t\int_{\R^d}\rho^N u^N(\Delta\phi + \nabla_x\DIV_x\phi) + 2\left(\nabla_x\sqrt{\rho^N}\otimes \sqrt{\rho^N}u^N\right):\nabla_x\phi\diff x\diff s,
        \end{split}
    \end{equation*}
    and the following energy bound is satisfied (in addition to the ones given in Theorem \ref{thm:existence_for_approximation})
    \begin{equation}\label{ineq:approx_rho_w_energy_ineq}
        \begin{split}
            \int_{\R^d}\rho^N(t, x)|w^N(t, x)|^2\diff x\leq  C(T, \|D^2K\|_\infty, \|\rho_0^N|u^N_0|^2\|_{L^\infty_t L^1_x}, \||x|^2\rho^N_0\|_{L^1_x}),
        \end{split}
    \end{equation}
    for all $t\in (0, T)$.
\end{thm}

\begin{proof}
    We only need to check the momentum conservation equation and the energy inequality. We start with the first one. Note that testing the equation \eqref{eq:approx_continuity} by $\nabla_x K(\cdot-y)$ implies that $\nabla_x K\star\rho^N$ is differentiable in time as
  \eq{\label{dtGradK}
    \frac{\diff}{\diff t}\int_{\R^d}\nabla_xK(x-y)\rho^N(t, y)\diff y = \int_{\R^d}\rho^N u^N(t, y)\nabla_y\nabla_x K(x-y)\diff y = - \int_{\R^d}\rho^N u^N(t, y)D^2 K(x-y)\diff y.
  }
    Hence, by \eqref{eq:approx_continuity} and the above
    \begin{equation}\label{eq:approx_rho_w_some_eq_1}
        \begin{split}
            &\int_0^t\int_{\R^d}\rho^N w^N\cdot\partial_t\phi\diff x\diff s = \int_0^t\int_{\R^d}\rho^N u^N\cdot\partial_t\phi\diff x\diff s + \int_0^t\int_{\R^d}\rho^N \nabla_x K\star\rho^N\cdot\partial_t\phi\diff x\diff s \\
            &= \int_0^t\int_{\R^d}\rho^N u^N\cdot\partial_t\phi\diff x\diff s + \int_0^t\int_{\R^d}\rho^N \partial_t\left(\nabla_x K\star\rho^N\cdot\phi\right)\diff x\diff s - \int_0^t\int_{\R^d}\rho^N\phi\cdot\partial_t\nabla_x K\star\rho^N\diff x\diff s\\
            &= \int_0^t\int_{\R^d}\rho^N u^N\cdot\partial_t\phi\diff x\diff s + \int_{\R^d}\rho^N(t, x)\nabla_xK\star\rho^N\cdot\phi(t, x)\diff x -  \int_{\R^d}\rho^N_0\nabla_xK\star\rho^N_0\cdot\phi(0, x)\diff x\\
            &\qquad -\int_0^t\int_{\R^d}\rho^N u^N\cdot\nabla_x(\nabla_x K\star\rho^N\cdot\phi)\diff x\diff s\\
            &\qquad+ \int_0^t\int_{\R^{d\times d}}\rho^N(s, x)\phi(s ,x)D^2K(x-y)\rho^N(s, y)u^N(s, y)\diff y\diff x\diff s\\
            &= \int_0^t\int_{\R^d}\rho^N u^N\cdot\partial_t\phi\diff x\diff s + \int_{\R^d}\rho^N(t, x)\nabla_xK\star\rho^N\cdot\phi(t, x)\diff x -  \int_{\R^d}\rho^N_0\nabla_xK\star\rho^N_0\cdot\phi(0, x)\diff x\\
            &\qquad -\int_0^t\int_{\R^d}\rho^N\nabla_x K\star\rho^N \otimes u^N : \nabla_x\phi\diff x\diff s\\
            &\qquad -\int_0^t\int_{\R^{d\times d}}\rho^N(x)\phi D^2K(x-y)(u^N( x) - u^N( y))\rho^N(y)\diff y\diff x\diff s.
        \end{split}
    \end{equation}
    Expressing the first and last term by \eqref{eq:approx_momentum}, i.e.
    \begin{equation*}
    \begin{split}
        &\int_0^t\int_{\R^d}\rho^N u^N\cdot\partial_t\phi\diff x\diff s -\int_0^t\int_{\R^{d\times d}}\rho^N(x)\phi D^2K(x-y)(u^N( x) - u^N( y))\rho^N(y)\diff y\diff x\diff s \\
        & = \int_{\R^d}\rho^N(t, x)u^N(t, x)\cdot\phi^N(t, x)\diff x - \int_{\R^d}\rho_0^N(x) u^N_0(x)\cdot \phi(0, x)\diff x\\
        &\qquad - \int_0^t\int_{\R^d}\rho^N u^N\otimes u^N : \nabla_x\phi - \frac{1}{N}\sqrt{\rho^N}\overline{\sqrt{\rho^N}\mathbb{D}u^n} : \nabla_x\phi \diff x\diff s,
    \end{split}
    \end{equation*}
   and using \eqref{eq:def_of_w_N} gives us desired formula \eqref{eq:approx_rho_w_momentum}. 
   
   To obtain the energy inequality, we simply use the formula \eqref{eq:def_of_w_N}, the triangle inequality, \eqref{eq:approx_mass_conservation}, \eqref{ineq:lipshitz_nabla_K} and the known bounds \eqref{ineq:approx_energy_ineq}, \eqref{ineq:approx_second_moment_ineq} to get
    \begin{equation*}
        \begin{split}
            &\int_{\R^d}\rho^N(t, x)|w^N(t, x)|^2\diff x \leq \frac{1}{2}\int_{\R^d}\rho^N|u^N|^2\diff x + \frac{1}{2}\int_{\R^d}\rho^N|\nabla_xK\star\rho^N|^2\diff x\\
            &\qquad\leq C(\|D^2K\|_\infty, \|\rho_0^N|u^N_0|^2\|_{L^\infty_t L^1_x,}) + \frac{1}{2}\int_{\R^d}\rho^N(t, x)\left|\int_{\R^d}C(1 + |x| + |y|)\rho^N(t, y)\diff y\right|^2\diff x\\
            &\qquad\quad= C(T, \|D^2K\|_\infty, \|\rho_0^N|u^N_0|^2\|_{L^1_x}, \||x|^2\rho_0^N\|_{L^1_x}).
        \end{split}
    \end{equation*}
\end{proof}

\begin{rem}
    Due to construction $w^N$ is not a suitable test function for equation \eqref{eq:approx_rho_w_momentum}, if it were, the kinetic energy associated with velocity $w^N$ would be bounded independently of $T$ and $K$.
\end{rem}

Before convergence, we once again need additional a priori bounds, similar to Lemma \ref{lem:bounds_for_lipshitz_distance}.

\begin{lem}\label{lem:bounds_for_lipshitz_distance_rho_w}
    Let $(\rho^N, w^N)$ be given by Theorem \ref{thm:existence_for_approximation_rho_w}, then
    \begin{enumerate}
        \item For every $\varphi \in C^\infty_c(\R^d)$, $\left\{\frac{d}{dt}\int_{\R^d}\varphi(x)\cdot\rho^N(t, x)\nabla_x K\star\rho^N(t, x)\diff x\right\}_{N\in\N}$ is bounded in $L^\infty(0, T)$ with a constant depending only on initial data and $\|\varphi\|_\infty$, $\|\nabla_x \varphi\|_{\infty}$,\label{bound:approx_Lipsh_density_w}
        \item For every $\varphi \in C^\infty_c(\R^d)$, $\left\{\frac{d}{dt}\int_{\R^d}\varphi(x) \cdot\rho^N w^N\diff x\right\}_{N\in\N}$ is bounded in $L^2(0, T)$ with a constant depending only on initial data and $\|\varphi\|_\infty$, $\|\nabla_x \varphi\|_{\infty}$.\label{bound:approx_Lipsh_momentum_w}
    \end{enumerate}
\end{lem}

\begin{proof}
    The second bound is proven the same way as in Lemma \ref{lem:bounds_for_lipshitz_distance}, therefore we skip the proof. For the first one, we use \eqref{dtGradK}
and \eqref{eq:approx_rho_w_continuity} to estimate
    \begin{equation*}
        \begin{split}
            &\left|\frac{d}{dt}\int_{\R^d}\varphi(x)\cdot\rho^N(t, x)\nabla_x K\star\rho^N(t, x)\diff x\right|\\
            &= \left|\int_0^t\int_{\R^d}\rho^N\partial_t(\varphi\cdot\nabla_x K\star\rho^N)\diff x\diff s + \int_0^t\int_{\R^d}\rho^N u^N\cdot\nabla_x(\varphi\cdot\nabla_x K\star\rho^N)\diff x\diff s\right|\\
            & \leq \left|\int_0^t\int_{\R^{d\times d}}\varphi(x)\rho^N(x)D^2K(x-y)(u^N(x) - u^N(y))\rho^N(y)\diff y\diff x\diff s\right|\\
            &\qquad +\left|\int_0^t\int_{\R^d}\rho^N\nabla_x K\star\rho^N\otimes u^N :\nabla_x\varphi\diff x\diff s\right|\\
            &\leq 2\|\varphi\|_\infty\|D^2K\|_\infty\|\rho^N\|_{L^\infty_t L^1_x}^{3/2}\|\rho^N|u^N|^2\|_{L^\infty_t L^1_x}^{1/2} + \|\nabla_x\varphi\|_\infty(\|\rho^N|w^N|^2\|_{L^\infty_t L^1_x} + \|\rho^N|u^N|^2\|_{L^\infty_t L^1_x}),
        \end{split}
    \end{equation*}
    which is bounded in $L^\infty(0, T)$ by \eqref{ineq:approx_energy_ineq}, \eqref{ineq:approx_rho_w_energy_ineq}.
\end{proof}

\begin{proof}[Proof of Theorem \ref{thm:main_existence} (Part $2$)]
    Looking at the proof of the Theorem \ref{thm:main_existence} (Part $1$) (specifically Lemma \ref{lem:measure_convergence_lemma}) we already know that there exist $\rho\in C([0, T]; (\mathcal{P}(\R^d), d_f))$, $m_u\in L^\infty(0, T; (\mathcal{M}(\R^d), d_f))$, $u\in L^\infty(0, T; L^2(\diff\rho_t))$ such that
    \begin{align*}
        \rho^N &\rightarrow \rho, &&\text{ in }C([0, T]; (\mathcal{P}(\R^d), d_f)),\\
        \rho^Nu^N =:m^N &\rightarrow m_u, &&\text{ in }L^\infty(0, T; (\mathcal{M}(\R^d), d_f)),\\
    \end{align*}
    and
    $$
    \diff m_{u,t}= u(t,\cdot)\diff\rho_t.
    $$
    As the limits are unique, the pair $(\rho, u)$ is the same one constructed in the first part of the proof. 
    
     In order to get analogous result  for the limit of  $(\rho^N ,w^N)$, apart from Lemma \ref{lem:bounds_for_lipshitz_distance_rho_w}, we also need to show that the first moment of $\rho^N w^N$  is bounded (to ensure tightness). Using H\"{o}lder's inequality, \eqref{ineq:approx_energy_ineq}, \eqref{ineq:approx_second_moment_ineq} and \eqref{ineq:lipshitz_nabla_K} we deduce
    \begin{equation*}
        \begin{split}
            &\int_{\R^d}|x|\rho^N(t, x)|w^N(t, x)|\diff x \leq \int_{\R^d}|x|\rho^N(t, x)|u^N(t, x)|\diff x + \int_{\R^d}|x|\rho^N(t, x)|\nabla_x K\star\rho^N(t, x)|\diff x \\
            &\leq \|\rho^N|u^N|^2\|^{1/2}_{L^\infty_t L^1_x}\left(\int_{\R^d}|x|^2\rho^N(t, x)\diff x\right)^{1/2} + C\int_{\R^d}|x|\rho^N(x)\int_{\R^d}(1 + |x| + |y|)\rho^N(y)\diff y\diff x\\
            &\leq C(T, \|D^2K\|_\infty, \|\rho_0^N|u_0^N|^2\|_{L^1_x}, \||x|^2\rho^N_0\|_{L^1_x}).
        \end{split}
    \end{equation*}
    Thus, we can imply that there exists $m_w\in L^\infty(0, T; (\mathcal{M}(\R^d), d_f))$, $\chi\in C([0, T]; (\mathcal{M}(\R^d), d_f))$, and $w\in L^\infty(0, T; L^2(\diff\rho_t))$, such that
    \begin{align*}
       \rho^Nw^N =:m^N_w &\rightarrow m_w, &&\text{ in }L^\infty(0, T; (\mathcal{M}(\R^d), d_f)),\\
       \rho^N\nabla_xK\star\rho^N &\rightarrow \chi, &&\text{ in }C([0, T]; (\mathcal{M}(\R^d), d_f)),
    \end{align*}
    and
    $$
    \diff m_{w,t}= w(t,\cdot)\diff\rho_t.
    $$
    At this point, following the argument in the first part of the proof of  Theorem \ref{thm:main_existence}, we can easily pass to the limit in the weak formulations \eqref{eq:approx_rho_w_continuity} and \eqref{eq:approx_rho_w_momentum}, where the measure $\nu$ is given by
    $$
    \nu := \mathrm{weak^*}\lim(\rho^Nw^N\otimes u^N) -  w\otimes u\diff\rho.
    $$
    Note that unlike for $\mu$, the sign of $\nu$ is undetermined.
    
    The last two points to verify are the compatibility condition \eqref{eq:final_compatibility_rho_w} and \eqref{ineq:final_energy_inequality_rho_w}. Note using equation \eqref{eq:def_of_w_N}, we may write
    \begin{align*}
        d_f(m_{w, t}, m_{u, t} + \chi(t))\leq d_f(m_{w, t}, \rho^N w^N\diff x) + d_f(\rho^Nu^N\diff x + \rho^N\nabla_x K\star\rho^N\diff x, m_{u, t} + \chi(t)).
    \end{align*}
    which, after passing to the limit on the right-hand side, gives us
    \begin{align}\label{eq:almost_comparibility}
    w(t)\diff\rho_t = u(t)\diff\rho_t + \chi(t),\text{ for a.e. }t\in(0, T).
    \end{align}
    To conclude, we need to identify $\chi$ with $(\nabla_x K\star\diff\rho)\diff\rho$. To do so fix $\phi\in C_0(\R^d)$ and $l\in\N$. We estimate
    \begin{equation*}
        \begin{split}
            &\left|\int_{\R^{d\times d}}\phi(x)\rho^N(t, x)\nabla_x K(x-y)\rho^N(t, y)\diff y\diff x - \int_{\R^{d\times d}}\phi(x)\nabla_x K(x-y)\diff\rho_t(y)\diff\rho_t(x)\right| \\
            &\leq \left|\int_{\R^{d\times d}}\phi(x)\rho^N(t, x)\nabla_x K(x-y)\rho^N(t, y)\diff y\diff x - \int_{\R^{d\times d}}\phi(x)\rho^N(t, x)(\nabla_x K(x-y)\wedge l)\rho^N(t, y)\diff y\diff x\right|\\
            &\, +\left|\int_{\R^{d\times d}}\phi(x)\rho^N(t, x)(\nabla_x K(x-y)\wedge l)\rho^N(t, y)\diff y\diff x  - \int_{\R^{d\times d}}\phi(x)(\nabla_x K(x-y)\wedge l)\diff\rho_t(y)\diff\rho_t(x)\right|\\
            &\, + \left|\int_{\R^{d\times d}}\phi(x)(\nabla_x K(x-y)\wedge l)\diff\rho_t(y)\diff \rho_t(x) - \int_{\R^{d\times d}}\phi(x)\nabla_x K(x-y)\diff\rho_t(y)\diff\rho_t(x)\right|\\
            &= I_1 + I_2 + I_3.
        \end{split}
    \end{equation*}
    Let us treat all the terms separately. We deduce by \eqref{ineq:lipshitz_nabla_K} and \eqref{ineq:approx_second_moment_ineq}
    \begin{equation*}
        \begin{split}
            I_1 &\leq \int_{|\nabla_x K(x-y)|\geq l}|\phi(x)||\nabla_x K(x-y) - l|\rho^N(t, x)\rho^N(t, y)\diff y\diff x\\
            &\leq 2C\int_{\{C(1 + |x - y|) \geq l\}}|\phi(x)|(1 + |x-y|)\rho^N(t, x)\rho^N(t, y)\diff y\diff x\\
            &\leq  C\|\phi\|_\infty\frac{1}{l}\int_{\R^d}(1 + |x-y|)^2\rho^N(t, x)\rho^N(t, y)\diff y\diff x\\
            &\leq \frac{C}{l},
        \end{split}
    \end{equation*}
    for a constant $C > 0$ independent of $N$. For $I_2$, we may argue similarly as in calculations following \eqref{tightness_of_product} (for fixed $l$ the function $\phi(x)(\nabla_xK(x-y)\wedge l)$ is bounded on $\R^d\times \R^d$), hence we skip the lengthy calculation, and deduce
    $$
    I_2 \rightarrow 0, \text{ as }N\to +\infty.
    $$
    To treat $I_3$ we may simply use Lebesgue's dominated convergence theorem, where the bound 
    $$
    |\phi(x)(\nabla_xK(x-y)\wedge l)|\leq C\|\phi\|_\infty(1 + |x-y|),
    $$
    is given in \eqref{ineq:lipshitz_nabla_K}, and the integrability of the bound is deduced from \eqref{ineq:approx_second_moment_ineq}. Hence, converging first with $N\to +\infty$, and then with $l\to +\infty$, we obtain
    $$
    \left|\int_{\R^{d\times d}}\phi(x)\rho^N(t, x)\nabla_x K(x-y)\rho^N(t, y)\diff y\diff x - \int_{\R^{d\times d}}\phi(x)\nabla_x K(x-y)\diff\rho_t(y)\diff\rho_t(x)\right|\rightarrow 0,
    $$
    as $N\to +\infty$, for any $\phi\in C_0(\R^d)$, which is enough to deduce $\chi(t) = (\nabla_x K\diff \rho_t)\diff\rho_t$ and in consequence (by \eqref{eq:almost_comparibility}) \eqref{eq:final_compatibility_rho_w}. 
    
    At last, to verify the energy bound we proceed similarly as in the proof of \eqref{ineq:final_energy_inequality}. By the linearity of the trace function and \eqref{eq:final_compatibility_rho_w} we may write
    \begin{equation*}
    \begin{split}
        &\frac{1}{2}\int_{\R^d}|w(t, x)|^2\diff \rho_t(x) + \frac{1}{2}\int_{\R^d}|u(t, x)|^2\diff \rho_t(x) + \mathrm{tr}(\nu_t(\R^d))\\
        & \quad= \frac{1}{2}\int_{\R^d}|w(t, x)|^2\diff\rho_t + \frac{1}{2}\int_{\R^d}|u(t, x)|^2\diff\rho_t +\lim_{N\to +\infty} \int_{\R^d}\mathrm{tr}(w^N\otimes u^N)\diff\rho^N_t - \int_{\R^d}\mathrm{tr}(w\otimes u)\diff\rho_t\\
        & \leq \frac{1}{2}\int_{\R^d}|w(t, x) - u(t, x)|^2\diff\rho_t + \liminf_{N\to +\infty}\int_{\R^d}\rho^N|w^N||u^N|\diff x\\
        &\quad = \frac{1}{2}\int_{\R^d}|\nabla_x K\star\diff\rho_t|^2\diff\rho_t + \liminf_{N\to +\infty}\left(\frac{1}{2}\int_{\R^d}\rho^N|w^N|^2\diff x + \frac{1}{2}\int_{\R^d}\rho^N|u^N|^2\diff x\right),
    \end{split}
    \end{equation*}
    which is bounded by $C(T, \|D^2K\|_\infty, \||u_0|^2\|_{L^1(\diff\rho_0)}, \||x|^2\|_{L^1(\diff\rho_0)})$, by the virtue of \eqref{ineq:lipshitz_nabla_K}, \eqref{bound_second_moment_in_approx}, \eqref{ineq:approx_second_moment_ineq}, \eqref{ineq:approx_energy_ineq}, \eqref{ineq:approx_rho_w_energy_ineq}, and the Proposition \ref{prop:lowersemicont_measures}.
\end{proof}

At this point, we have proven the statement of Theorem \ref{thm:main_existence} for initial datum $u_0\in C_b(\R^d)$ and $\rho_0\in\mathcal{P}(\R^d)$ with a bounded second moment. As a last step we want to lift this restriction. Take a sequence
\begin{align*}
   C_b(\R^d) \ni u_0^M \rightarrow u_0,\text{ strongly in }L^2(\diff\rho_0),
\end{align*}
where $\rho_0\in \mathcal{P}(\R^d)$ with a bounded second moment. For such $(u_0^M, \rho_0)$, by a previous step, we know that the statement of Theorem \ref{thm:main_existence} holds, and in particular there exists a sequence of measure solutions $(\rho^M, u^M, \mu^M)$ in the sense of Definition \ref{def:measure_solution}.

\begin{proof}[Proof of Theorem \ref{thm:main_existence} (Part $3$)]
    It is quite clear that since
    $$
    \int_{\R^d}|u_0^M|^2\diff\rho_0 \rightarrow \int_{\R^d}|u_0|^2\diff\rho_0,
    $$
    the arguments that were used to converge with $N\to +\infty$ in the previous step still hold. That is, we know that
    \begin{align*}
        \rho^M &\rightarrow \rho, &&\text{ in }C([0, T]; (\mathcal{P}(\R^d), d_f)),\\
        u^M\diff\rho^M =:m^M &\rightarrow m_u, &&\text{ in }L^\infty(0, T; (\mathcal{M}(\R^d), d_f)),\\
    \end{align*}
    and
    $$
    \diff m_{u,t}= u(t,\cdot)\diff\rho_t,
    $$
    as well as
    \begin{align*}
       w^M\diff\rho^M =:m^M_w &\rightarrow m_w, &&\text{ in }L^\infty(0, T; (\mathcal{M}(\R^d), d_f)),\\
       (\nabla_xK\star\diff\rho^M)\diff\rho^M &\rightarrow (\nabla_xK\star\diff\rho)\diff\rho, &&\text{ in }C([0, T]; (\mathcal{M}(\R^d), d_f)),
    \end{align*}
    and
    $$
    \diff m_{w,t}= w(t,\cdot)\diff\rho_t.
    $$
    Hence, we can converge in the weak formulations and the energy inequalities in the same way as before. Let us briefly explain how to handle the terms with the concentration measure. From the energy inequality \eqref{ineq:final_energy_inequality} satisfied by $(\rho^M, u^M, \mu^M)$, and Proposition \ref{prop:bound_symmetric_measure_trace} $\{\mu^M\}_{M\in\N}$ is bounded in $L^\infty(0, T; \M^+(\R^d; \R^{d\times d}_{\mathrm{sym}}))$. Thus, by the Banach--Alaoglu theorem, there exists $\mu_1\in L^\infty(0, T; \M^+(\R^d; \R^{d\times d}_{\mathrm{sym}}))$, such that (up to the subsequence which we do not relabel)
    \begin{align*}
        \mu^M \wstar \mu_1,\text{ weakly* in }L^\infty(0, T; \M^+(\R^d; \R^{d\times d}_{\mathrm{sym}})),
    \end{align*}
    which is enough to converge in the term
    $$
    \int_0^t\int_{\R^d}\nabla_x\phi :\diff\mu^M_\tau(x)\diff \tau \rightarrow \int_0^t\int_{\R^d}\nabla_x\phi :\diff\mu_{1, \tau}(x)\diff \tau.
    $$
    Similarly as before we may define
    $$
    \mu_2 := \mathrm{weak^*}\lim(u^M\otimes u^M\diff\rho^M) -  u\otimes u\diff\rho,
    $$
    which helps us cover the convection term, and is a new concentration measure appearing from this step of the approximation. With this, after denoting $\mu := \mu_1 + \mu_2$, we obtain in a weak formulation a term
    \begin{align*}
        \int_0^t\int_{\R^d}\nabla_x\phi :\diff\mu_{1, \tau}(x)\diff \tau + \int_0^t\int_{\R^d}\nabla_x\phi :\diff\mu_{2, \tau}(x)\diff \tau = \int_0^t\int_{\R^d}\nabla_x\phi :\diff\mu_{\tau}(x)\diff \tau.
    \end{align*}
    Here $\mu$ is the full concentration measure combining the effects coming from the convective term, as well as from the approximate sequence $\{\mu^M\}_{M\in \N}$. Moreover, following the argument leading to \eqref{ineq:proof_N_conv_energy}
    \begin{equation*}  
        \begin{split}
            &\frac{1}{2}\int_{\R^d}|u(t, x)|^2\diff\rho_t(x) + \frac{1}{2}\mathrm{tr}(\mu_t(\R^d)) = \frac{1}{2}\int_{\R^d}|u(t, x)|^2\diff\rho_t(x) + \frac{1}{2}\mathrm{tr}(\mu_{1,t}(\R^d)) + \frac{1}{2}\mathrm{tr}(\mu_{2,t}(\R^d)) \\
            &\phantom{=}\leq \frac{1}{2}\int_{\R^d}|u(t, x)|^2\diff\rho_t(x)  + \frac{1}{2}\liminf_{M\to +\infty} \left(\mathrm{tr}(\mu^M_{t}(\R^d)) + \int_{\R^d}\mathrm{tr}(u^M\otimes u^M)\diff\rho^M_t(x)\right) -  \frac{1}{2}\int_{\R^d}\mathrm{tr}(u\otimes u)\diff \rho_t(x)\\
            &\phantom{=}= \frac{1}{2}\liminf_{M\to +\infty}\left(\int_{\R^d}|u^M|^2\diff\rho^M(x) + \mathrm{tr}(\mu^M_{t}(\R^d))\right) \leq e^{4t\|D^2K\|_\infty}\frac{1}{2}\int_{\R^d}|u_0|^2\diff\rho_0.
        \end{split}
    \end{equation*}
    which proves \eqref{ineq:final_energy_inequality} for general $u_0$. Analogously one can show that \eqref{ineq:energy_inequality_for_symmetric} and \eqref{ineq:final_energy_inequality_rho_w} hold as well.
\end{proof}

\section{Stability with respect to the initial datum, and weak-strong uniqueness}\label{Sec:5}
In this Section we prove Theorem \ref{thm:stability_initial_datum}. Our technique  is similar to \cite{figali2019arigorous} and \cite{CC21}, where authors show the convergence of  mesoscopic and microscopic approximations of the system \eqref{eq:main_sys_1} for a  scalar communication kernel, in the relative entropy distance. Before the proof, we mention an important proposition proven by Figali and Kang in \cite{figali2019arigorous}, which allows us to relate the distance between densities to a distance between the velocities. The fact, that the right-hand side of the inequality below is integrated with respect to the weak density is crucial.
\begin{prop}[Lemma 5.2, \cite{figali2019arigorous}]\label{prop:Kang_Figali_ineq}
    Suppose the assumptions of Theorem \ref{thm:stability_initial_datum} hold. Then, the following inequality holds
    \begin{align*}
        W_2^2(\rho^n_t, r_t) \leq C(\|v\|_{L^\infty(0, T; W^{1,\infty}(\R^d))})e^T\int_{\R^d}&|u^n(t, x) - v(t, x)|^2\diff\rho_t(x)\\
        &+ C(T, \|v\|_{L^\infty((0, T)\times\R^d)})\|\rho^n_0 - r_0\|_{TV} + C\||x|^2(\rho^n_0 - r_0)\|_{TV}.
    \end{align*}
\end{prop}

\begin{rem}
    Note that the statement of Proposition \ref{prop:Kang_Figali_ineq} slightly differs from Lemma 5.2 in \cite{figali2019arigorous}. The change comes from the fact that in a bounded domain (which is the case in \cite{figali2019arigorous}), one has
    $$
    W_2^2(\rho_1, \rho_2) \leq \frac{d}{8}\|\rho_1 - \rho_2\|_{L^1}.
    $$
    In the case of an unbounded domain, which is our setting, one needs to adapt the inequality above. For example, by Proposition 7.10 \cite{Villani2003Topics}, one has
    $$
    W^2_2(\rho_1, \rho_2) \leq C\||x|^2(\rho_1 - \rho_2)\|_{TV}.
    $$
\end{rem}

\begin{proof}[Proof of the Theorem \ref{thm:stability_initial_datum}]
    We want to test \eqref{eq:final_mass_conservation} by $|v|^2$ and \eqref{eq:final_momentum_conservation} by $-2v$. This is not yet justified as we only know that $v\in L^\infty(0, T; W^{1,\infty}(\R^d))$. However, equation $ \eqref{eq:main_sys_strong}_2 $, and the Sobolev embeddings imply $v\in C^1([0, T]\times \R^d)$, making it an eligible test function in the weak formulation of system \eqref{eq:main_sys_1}.  Although $v$ does not necessarily have a compact support in $[0, T)$, since we consider the equation for a fixed $t < T$, we can extend it arbitrarily on $[t, T)$, and it will not affect our considerations. Indeed, we are interested only with the behavior of $v$ on $[0, t]$. 
    After testing and adding the results together to obtain
    \begin{equation*}
    \begin{split}
        &\int_{\R^d}|v(t, x)|^2 - 2v(t, x)\cdot u^n(t, x)\diff\rho^n_t(x) - \int_{\R^d}|v_0|^2 - 2v_0\cdot u^n_0\diff\rho^n_0(x)\\
        &\quad= 2\int_0^t\int_{\R^d}(v-u^n)\cdot\partial_t v\diff\rho^n_s(x)\diff s + 2\int_0^t\int_{\R^d}u^n\otimes(v-u^n) : \nabla_x v\diff \rho^n_s(x)\diff s\\
        &\qquad -2\int_0^t\int_{\R^d}\nabla_x v:\diff\mu^n_s(x)\diff s + \int_0^t\int_{\R^{d\times d}}v(s, x)D^2K(x - y)(u^n(s, x) - u^n(s, y))\diff\rho^n_s(y)\diff\rho^n_s(x)\diff s.
    \end{split}
    \end{equation*}
Adding further equation \eqref{ineq:energy_inequality_for_symmetric} to the above we get
    \begin{equation}\label{eq:some_eq_weak_strong_1}
    \begin{split}
        &\int_{\R^d}|v(t, x) - u^n(t, x)|^2\diff\rho^n_t(x) - \int_{\R^d}|v_0(x) - u^n_0(x)|^2\diff\rho^n_0(x) + \mathrm{tr}(\mu^n_t(\R^d))\\
        &\qquad + \int_0^t\int_{\R^{d\times d}}(u^n(s, x) - u^n(s, y))D^2K(x - y)(u^n(s, x) - u^n(s, y))\diff\rho^n_s(y)\diff\rho^n_s(x)\diff s\\
        &\leq 2\int_0^t\int_{\R^d}(v-u^n)\cdot\partial_t v\diff\rho^n_s(x)\diff s + 2\int_0^t\int_{\R^d}u^n\otimes(v-u^n) : \nabla_x v\diff \rho^n_s(x)\diff s\\
        &\qquad -2\int_0^t\int_{\R^d}\nabla_x v:\diff\mu^n_s(x)\diff s + 2\int_0^t\int_{\R^{d\times d}}v(s, x)D^2K(x - y)(u^n(s, x) - u^n(s, y))\diff\rho^n_s(y)\diff\rho^n_s(x)\diff s.
    \end{split}
    \end{equation}
Using equation \eqref{eq:main_sys_strong}$_2$ for $v$, we may write
\begin{equation*}
    \begin{split}
        &\int_0^t\int_{\R^d}(v-u^n)\cdot\partial_t v\diff\rho^n_s(x)\diff s + \int_0^t\int_{\R^d}u^n\otimes(v-u^n) : \nabla_x v\diff \rho^n_s(x)\diff s\\
        & = -\int_0^t\int_{\R^d}(v-u^n)\cdot(v\cdot\nabla_x)v\diff\rho^n_s(x)\diff s + \int_0^t\int_{\R^d}u^n\otimes(v-u^n) : \nabla_x v\diff \rho^n_s(x)\diff s\\
        &\qquad - \int_0^t\int_{\R^{d\times d}}(v(s, x) - u^n(s, x))D^2K(x-y)(v(s, x) - v(s, y))\diff r_s(y)\diff\rho^n_s(x)\diff s\\
        & = \int_0^t\int_{\R^d}(u^n - v)\otimes(v-u^n) : \nabla_x v\diff \rho^n_s(x)\diff s\\
        &\qquad - \int_0^t\int_{\R^{d\times d}}(v(s, x) - u^n(s, x))D^2K(x-y)(v(s, x) - v(s, y))\diff r_s(y)\diff\rho^n_s(x)\diff s.
    \end{split}
\end{equation*}
Putting it into \eqref{eq:some_eq_weak_strong_1} we deduce the relative entropy inequality in the form:
\begin{equation*}
    \begin{split}
        &\int_{\R^d}|v(t, x) - u^n(t, x)|^2\diff\rho^n_t(x) - \int_{\R^d}|v_0(x) - u^n_0(x)|^2\diff\rho^n_0(x) + \mathrm{tr}(\mu^n_t(\R^d))\\
        &\qquad + \int_0^t\int_{\R^{d\times d}}(u^n(s, x) - u^n(s, y))D^2K(x - y)(u^n(s, x) - u^n(s, y))\diff\rho^n_s(y)\diff\rho^n_s(x)\diff s\\
        &\leq2\int_0^t\int_{\R^d}(u^n - v)\otimes(v-u^n) : \nabla_x v\diff \rho^n_s(x)\diff s  -2\int_0^t\int_{\R^d}\nabla_x v:\diff\mu^n_s(x)\diff s\\
        &\qquad - 2\int_0^t\int_{\R^{d\times d}}(v(s, x) - u^n(s, x))D^2K(x-y)(v(s, x) - v(s, y))\diff r_s(y)\diff\rho^n_s(x)\diff s\\
        &\qquad + 2\int_0^t\int_{\R^{d\times d}}v(s, x)D^2K(x - y)(u^n(s, x) - u^n(s, y))\diff\rho^n_s(y)\diff\rho^n_s(x)\diff s.
    \end{split}
\end{equation*}
We now need to estimate the terms on the right-hand side.
To treat the first two terms we use boundedness of $\nabla_x v$ and  Proposition \ref{prop:bound_symmetric_measure_trace}. We have
\begin{align*}
    &2\int_0^t\int_{\R^d}(u^n - v)\otimes(v-u^n) : \nabla_x v\diff \rho^n_s(x)\diff s  -2\int_0^t\int_{\R^d}\nabla_x v:\diff\mu^n_s(x)\diff s\\
    &\leq 2\|\nabla_x v\|_\infty\left(\int_0^t\int_{\R^d}|u^n - v|^2\diff\rho^n_s(x)\diff s +\int_0^t|\mu^n_s|(\R^d)\diff s\right)\\
    &\leq C\|\nabla_x v\|_\infty \left(\int_0^t\int_{\R^d}|u^n - v|^2\diff\rho^n_s(x)\diff s +\int_0^t\mathrm{tr}(\mu^n_s(\R^d))\diff s\right).
\end{align*}

To treat the last two terms we first split
\begin{equation*}
    \begin{split}
    &- 2\int_0^t\int_{\R^{d\times d}}(v(s, x) - u^n(s, x))D^2K(x-y)(v(s, x) - v(s, y))\diff r_s(y)\diff\rho^n_s(x)\diff s\\
    &+ 2\int_0^t\int_{\R^{d\times d}}v(s, x)D^2K(x - y)(u^n(s, x) - u^n(s, y))\diff\rho^n_s(y)\diff\rho^n_s(x)\diff s\\
    & \quad = - 2\int_0^t\int_{\R^{d\times d}}(v(s, x) - u^n(s, x))D^2K(x-y)(v(s, x) - v(s, y))\diff r_s(y)\diff\rho^n_s(x)\diff s\\
    &\qquad + 2\int_0^t\int_{\R^{d\times d}}(v(s, x) - u^n(s, x))D^2K(x-y)(v(s, x) - v(s, y))\diff \rho^n_s(y)\diff\rho^n_s(x)\diff s\\
    &\qquad - 2\int_0^t\int_{\R^{d\times d}}(v(s, x) - u^n(s, x))D^2K(x-y)(v(s, x) - v(s, y))\diff \rho^n_s(y)\diff\rho^n_s(x)\diff s\\
    &\qquad + 2\int_0^t\int_{\R^{d\times d}}(v(s, x) - u^n(s, x))D^2K(x - y)(u^n(s, x) - u^n(s, y))\diff\rho^n_s(y)\diff\rho^n_s(x)\diff s\\
    &\qquad + 2\int_0^t\int_{\R^{d\times d}}u^n(s, x)D^2K(x - y)(u^n(s, x) - u^n(s, y))\diff\rho^n_s(y)\diff\rho^n_s(x)\diff s\\
    &\quad = I_1 + I_2 + I_3 +I_4 + I_5.
    \end{split}
\end{equation*}
Now, by Young's inequality we can estimate
\begin{equation*}
    \begin{split}
    & I_1 + I_2\\
    &=- 2\int_0^t\int_{\R^{d\times d}}(v(s, x) - u^n(s, x))D^2K(x-y)(v(s, x) - v(s, y))\diff r_s(y)\diff\rho^n_s(x)\diff s\\
    &\qquad + 2\int_0^t\int_{\R^{d\times d}}(v(s, x) - u^n(s, x))D^2K(x-y)(v(s, x) - v(s, y))\diff \rho^n_s(y)\diff\rho^n_s(x)\diff s\\
    &\leq 4\|D^2K\|_{W^{1,\infty}}\|v\|_{W^{1,\infty}}\int_0^t\int_{\R^d}|u^n - v|d_f(\rho^n_s, r_s)\diff\rho^n_s(x)\diff s\\
    &\leq C\|D^2K\|_{W^{1,\infty}}\|v\|_{W^{1,\infty}}\left(\int_0^t\int_{\R^d}|u^n - v|^2\diff\rho^n_s(x)\diff s + \int_0^t W_2^2(\rho^n_s, r_s)\diff s\right),
    \end{split}
\end{equation*}
where we used $d_f(\rho^n_s, r_s)
\leq W_1(\rho^n_s, r_s)\leq W_2(\rho^n_s, r_s)
$, see Theorem 5.9 in \cite{Santambrogio}.
Moving forward, using the fact that $\diff\rho^n_s$ is a probability measure, and the H\"{o}lder inequality 
\begin{equation*}
    \begin{split}
    &I_3 + I_4\\
    & = - 2\int_0^t\int_{\R^{d\times d}}(v(s, x) - u^n(s, x))D^2K(x-y)(v(s, x) - v(s, y))\diff \rho^n_s(y)\diff\rho^n_s(x)\diff s\\
    &\qquad + 2\int_0^t\int_{\R^{d\times d}}(v(s, x) - u^n(s, x))D^2K(x - y)(u^n(s, x) - u^n(s, y))\diff\rho^n_s(y)\diff\rho^n_s(x)\diff s\\
    & = 2\int_0^t\int_{\R^{d\times d}}(v(s, x) - u^n(s, x))D^2K(x-y)((u^n(s, x) - v(s, x)) - (u^n(s, y) - v(s, y)))\diff \rho^n_s(y)\diff\rho^n_s(x)\diff s\\
    & \leq 2\|D^2K\|_\infty\int_0^t\int_{\R^d}|u^n - v|^2\diff\rho^n_s(x)\diff s + 2\|D^2K\|_\infty\int_0^t\left(\int_{\R^d}|u^n - v|\diff\rho^n_s(x)\right)^2\diff s\\
    &\leq 4\|D^2K\|_\infty\int_0^t\int_{\R^d}|u^n - v|^2\diff\rho^n_s(x)\diff s.
    \end{split}
\end{equation*}
At last, the fact that $D^2K$ is an even implies
$$
I_5 = \int_0^t\int_{\R^{d\times d}}(u^n(s, x) - u^n(s, y))D^2K(x - y)(u^n(s, x) - u^n(s, y))\diff\rho^n_s(y)\diff\rho^n_s(x)\diff s,
$$
which cancels out with the same term on the left-hand side of 
\eqref{eq:some_eq_weak_strong_1}. 
Combining all the obtained bounds we conclude
\begin{equation*}
\begin{split}
    &\int_{\R^d}|v(t, x) - u^n(t, x)|^2\diff\rho^n_t(x) - \int_{\R^d}|v_0(x) - u^n_0(x)|^2\diff\rho^n_0(x)\\
    &\qquad\leq C\|D^2K\|_{W^{1,\infty}}\|v\|_{W^{1,\infty}}\left(\int_0^t\int_{\R^d}|u^n - v|^2\diff\rho^n_s(x)\diff s + \int_0^t W_2^2(\rho^n_s, r_s)\diff s + \int_0^t\mathrm{tr}(\mu^n_s(\R^d))\diff s\right),
\end{split}
\end{equation*}
which added to the inequality from Proposition \ref{prop:Kang_Figali_ineq} gives
\begin{equation*}
\begin{split}
    &\int_{\R^d}|v(t, x) - u^n(t, x)|^2\diff\rho^n_t(x) + \mathrm{tr}(\mu^n_t(\R^d)) + W_2^2(\rho^n_t, r_t)\\
    &\qquad\leq C(T, \|D^2K\|_{W^{1,\infty}},\|v\|_{W^{1,\infty}})\Bigg(\int_0^t\int_{\R^d}|u^n - v|^2\diff\rho^n_s(x)\diff s + \int_0^t W_2^2(\rho^n_s, r_s)\diff s + \int_0^t\mathrm{tr}(\mu^n_s(\R^d))\diff s \\
    &\qquad\qquad + \int_{\R^d}|v_0(x) - u^n_0(x)|^2\diff\rho^n_0(x) + \|\rho^n_0 - r_0\|_{TV} + \||x|^2(\rho^n_0 - r_0)\|_{TV}\Bigg).
\end{split}
\end{equation*}
Using Gr\"{o}nwall's inequality we obtain the thesis.
\end{proof}

\appendix

\section{Existence of approximate solutions}\label{appendix:existence_for_approximation}

In this section we present the proof of Theorem \ref{thm:existence_for_approximation}. As mentioned earlier, the proof is a modification of  one given in \cite{mucha2025construction}, hence we only give a brief description of the steps that are different in our case. For the convenience of the reader, we first recall the existence theorem given in \cite{mucha2025construction}, which will serve as a comparison.

\begin{prop}[Theorem 2.3, \cite{mucha2025construction}]\label{prop:maja_existene_approx}
    Consider the equation
    \begin{align*}
        \left\{\begin{array}{ll}
             \partial_t\rho + \DIV_x(\rho u) = 0,  \\
             \partial_t(\rho u) + \DIV_x(\rho u\otimes u) -\DIV_x(\rho\mathbb{D}u) + \rho\nabla_x(V\star\rho) = 0,
        \end{array}
        \right.
    \end{align*}
    for 
    $$
    V := \frac{1}{|x|^\alpha} + \frac{|x|^2}{2},\qquad \alpha\in (0, 2),
    $$
    and initial data
    $$
    \rho(0, x) := \rho_0(x),\quad m(0, x) = (\rho u)(0, x) := m_0(x), \text{ such that }\rho_0\geq 0,\, \sqrt{\rho_0}\in W^{1,2}(\R^d),
    $$
    satisfying
    \begin{align*}
        \int_{\R^d}\rho_0 F(|u_0|)\diff x + \int_{\R^d\times\R^d}F(|x - y|)\rho_0(x)\rho_0(y)\diff x\diff y < +\infty,
    \end{align*}
    for
    $$
    F(z) := \frac{1 + z^2}{2}\ln(1 + z^2).
    $$
    Then, there exists a solution to the problem in the sense that
    \begin{align*}
        \rho\in L^\infty(0, T; L^1(\R^d)),&\quad \sqrt{\rho}\in L^\infty(0, T: W^{1,2}(\R^d)),\\
        \sqrt{\rho}u\in L^\infty(0, T; L^2(\R^d)),&\quad \overline{\rho \mathbb{D}u}\in L^2(0, T; W^{-1, 1}(\R^d)),\\
        \int_0^T\int_{\R^d\times\R^d}|x-y|^2&\rho(x)\rho(y)\diff x\diff y < +\infty,
    \end{align*}
    and for any $\varphi\in C^\infty_0([0, T)\times \R^d; \R)$, and $\psi\in C^\infty_0([0, T)\times \R^d; \R^d)$ we have
    \begin{align*}
        -\int_0^T\int_{\R^d}\rho\partial_t\varphi\diff x\diff t - \int_0^T\int_{\R^d}\rho u\cdot\nabla_x\varphi\diff x\diff t  = \int_{\R^d}\rho_0(x)\varphi(0, x)\diff x,
    \end{align*}
    and
    \begin{multline*}
        -\int_{\R^d}m_0(x)\cdot\psi(0, x)\diff x - \int_0^T\int_{\R^d}\sqrt{\rho}\sqrt{\rho}u\cdot\partial_t\psi\diff x\diff t - \int_0^T\int_{\R^d}(\sqrt{\rho}u\otimes \sqrt{\rho}u):\nabla_x\psi\diff x\diff t\\
        +\langle\overline{\rho \mathbb{D}u}, \nabla_x\psi\rangle + \int_0^T\int_{\R^d}\rho\nabla_x(V\star\rho)\cdot\psi\diff x\diff t = 0,
    \end{multline*}
    where one defines
    $$
    \langle\overline{\rho \mathbb{D}u}, \nabla_x\psi\rangle = -\int_0^T\int_{\R^d}\rho u\cdot(\Delta\psi + \nabla_x\DIV_x\psi) + 2(\sqrt{\rho}\otimes \sqrt{\rho}u) : \nabla_x\psi \diff x\diff t.
    $$
    Moreover, for any $t\geq 0$
    $$
    \int_{\R^d}\rho(t, x)\diff x = \int_{\R^d}\rho_0(x),
    $$
    and there exist functions $\overline{\overline{\sqrt{\rho}\mathbb{D}u}},\, \overline{\overline{\sqrt{\rho}(\nabla_x u - \nabla_x^Tu)}}\in L^2(0, T; L^2(\R^d))$, such that
    $$
    \sqrt{\rho}\overline{\overline{\sqrt{\rho}\mathbb{D}u}} = \overline{\rho\mathbb{D}u},\quad  \sqrt{\rho}\overline{\overline{\sqrt{\rho}(\nabla_x u - \nabla_x^Tu)}} = \overline{\rho(\nabla_x u - \nabla_x^Tu)},\quad \text{ a.e.,}
    $$
    where $\overline{\rho(\nabla_x u - \nabla_x^Tu)}$ is defined similarly to $\overline{\rho\mathbb{D}u}$, which satisify the following energy estimates
    \begin{enumerate}
        \item the energy estimate
        \begin{multline}\label{ineq:maja_basic_energy}
            \sup_{t\geq 0}\frac{1}{2}\int_{\R^d}\rho|u|^2\diff x + \int_0^\infty\int_{\R^d}\left|\overline{\overline{\sqrt{\rho}\mathbb{D}u}}\right|^2\diff x\diff t + \sup_{t\geq 0}\int_{\R^d\times\R^d}V(x-y)\rho(t, x)\rho(t, y)\diff x\diff y\\
            \leq \frac{1}{2}\int_{\R^d}\rho_0|u_0|^2\diff x + \int_{\R^d\times\R^d}V(x-y)\rho_0(x)\rho_0(y)\diff x\diff y.
        \end{multline}
        \item the Bresch--Dejardins estimate
        \begin{align}\label{ineq:maja_bresch_dejardins}
            \sup_{t\in [0, T]}\int_{\R^d}|\nabla_x\sqrt{\rho}|^2\diff x + \int_0^T\int_{\R^d}\left|\overline{\overline{\sqrt{\rho}(\nabla_x u - \nabla_x^Tu)}}\right|^2\diff x\diff t \leq C(T).
        \end{align}
        \item the Mellet--Vasseur estimate
        \begin{align}\label{ineq:maja_mellet_vasseur}
            \sup_{t\in [0, T]}\int_{\R^d}\rho F(|u|)\diff x + \sup_{t\in [0, T]}\int_{\R^d\times \R^d}F(|x - y|)\rho(t, x)\rho(t, y)\diff x\diff y \leq C(T).
        \end{align}
    \end{enumerate}
\end{prop}

As one can see, the difference between Theorem \ref{thm:existence_for_approximation} and Proposition \ref{prop:maja_existene_approx} is the form of the non-local term and its consequences on the energy estimates. The approximating system for the proof of the Proposition \ref{prop:maja_existene_approx} is given by
\begin{align*}
    \left\{\begin{array}{ll}
        \partial_t\rho + \DIV_x(\rho u) = \varepsilon\Delta_x\rho,   \\
        \partial_t(\rho u) + \DIV_x(\rho u\otimes u) - \DIV_x(\rho \mathbb{D}u) + \rho\nabla_x(V_L\star\rho)\\
        \qquad= -r_0 u - r_1\rho|u|^2u + \kappa\rho\nabla_x\left(\frac{\Delta_x\sqrt{\rho}}{\sqrt{\rho}}\right) - \varepsilon\nabla_x\rho\cdot\nabla_x u - \nu\Delta^2_xu + \eta\nabla_x\rho^{-6} +\delta\rho\nabla_x\Delta^3_x\rho,
    \end{array}
    \right.
\end{align*}
while in our case it is
\begin{align}\label{eq:appendix_approximate_system}
    \left\{\begin{array}{ll}
        \partial_t\rho + \DIV_x(\rho u) = \varepsilon\Delta_x\rho,   \\
        \partial_t(\rho u) + \DIV_x(\rho u\otimes u) - \DIV_x(\rho \mathbb{D}u) +\rho\int_{\R^d}D^2K_L(x-y) (u(x) - u(y))\rho(y)\diff y\\
        \qquad= -r_0 u - r_1\rho|u|^2u + \kappa\rho\nabla_x\left(\frac{\Delta_x\sqrt{\rho}}{\sqrt{\rho}}\right) - \varepsilon\nabla_x\rho\cdot\nabla_x u - \nu\Delta^2_xu + \eta\nabla_x\rho^{-6} +\delta\rho\nabla_x\Delta^3_x\rho,
    \end{array}
    \right.
\end{align}
where $V_L$ and $D^2K_L$ are functions $V, D^2K$, respectively, truncated to a periodic torus of size $L$ respectively. To get a better understanding of this technique, we briefly explain the need for all the parameters in the approximation (cf \cite[Table 1]{mucha2025construction})
\begin{itemize}
    \item The terms with $\varepsilon > 0$, are needed to perform standard construction of the strong solution at the level of Galerkin approximation.
    \item The terms with $\nu, \eta, \delta > 0$ give us enough density regularity to test the equation by $\nabla_x\log\rho$ and obtain the Bresch--Dejardins estimates .
    \item The terms with $\kappa, r_0, r_1 > 0$ gives us enough density and velocity regularity to renormalize the momentum equation.
    \item $L > 0$ allows us to consider a simpler case of bounded, periodic domain, and later expand the torus.
\end{itemize}

Let us now discuss the differences in the main steps of the approximation procedure in comparison to \cite{mucha2025construction}.
\begin{enumerate}
    \item The first difference in our case is in the stability of the non-local term. Not to repeat ourselves we refer to the discussion around \eqref{thightness_of_product_sequence}, where one can see that it is enough to have
    $$
    \rho_n\to \rho\quad\text{ strongly in }L^1_{t,x},\qquad \rho_n u_n\to \rho u\quad\text{ strongly in }L^1_{t,x},
    $$
    at any approximation step to properly converge in the non-local term. It is clearly given, when looking at the arguments in \cite{mucha2025construction}.
    
    \item The second difference is in the derivation of the basic energy inequality, which is used during he Galerkin approximation, and to obtain fundamental estimates (cf. \cite[Subsection 3.2, 3.3]{mucha2025construction}). As always, it is done by multiplying the momentum equation by the velocity, thus our goal is to show the estimate for
    \begin{align*}
        \left|\int_0^T\int_{\R^d\times\R^d}\rho(t, x)u(t, x)D^2K(x-y)(u(t, x) - u(t, y))\rho(t, y)\diff y\diff x\diff t\right|.
    \end{align*}
    Whenever $D^2K\in C_b(\R^d)$ only, then we may use mass conservation, and H\"{o}lder's inequality to get
    \begin{equation*}
        \begin{split}
            &\left|\int_0^T\int_{\R^d\times\R^d}\rho(t, x)u(t, x)D^2K(x-y)(u(t, x) - u(t, y))\rho(t, y)\diff y\diff x\diff t\right|\\
            &\leq \|D^2K\|_\infty \int_0^T\int_{\R^d}\rho|u|^2\diff x\diff t + \|D^2K\|_\infty\int_0^T\left(\int_{\R^d}\rho|u|\diff x\right)^2\diff t\\
            &\leq 2\|D^2K\|_\infty \int_0^T\int_{\R^d}\rho|u|^2\diff x\diff t,
        \end{split}
    \end{equation*}
    which allows us to use Gr\"{o}nwall's inequality (compare with the left-hand side of \eqref{ineq:maja_basic_energy}) and leads to an exponent term in \eqref{ineq:final_energy_inequality}. For a symmetric
    $D^2K$, by a change of variables, we may see that
    \begin{multline*}
        \int_0^T\int_{\R^d\times\R^d}\rho(t, x)u(t, x)D^2K(x-y)(u(t, x) - u(t, y))\rho(t, y)\diff y\diff x\diff t \\
        = \int_0^T\int_{\R^d\times\R^d}\rho(t, y)u(t, y)D^2K(x-y)(u(t, y) - u(t, x))\rho(t, x)\diff y\diff x\diff t,
    \end{multline*}
    thus, whenever $D^2K$ is positive definite
    \begin{multline*}
        \int_0^T\int_{\R^d\times\R^d}\rho(t, x)u(t, x)D^2K(x-y)(u(t, x) - u(t, y))\rho(t, y)\diff y\diff x\diff t\\
        = \frac{1}{2}\int_0^T\int_{\R^d\times\R^d}\rho(t, x)\rho(t, y)(u(t, x) - u(t, y))D^2K(x-y)(u(t, x) - u(t, y))\diff y\diff x\diff t \geq 0,
    \end{multline*}
    which leads to the energy term without the use of Gr\"{o}nwall's, and without the exponent term in \eqref{ineq:energy_inequality_for_symmetric}.  We note that this is enough to let the dimension of the Galerkin approximation go to $+\infty$.
    \item The third difference is in derivation of the Bresch--Dejardins estimate (cf. \cite[Section 4, Appendix A]{mucha2025construction}). This estimate is used to converge with $\varepsilon, \nu, \delta, \eta \to 0^+$ in \eqref{eq:appendix_approximate_system}. In fact, this is a crucial step to obtain an augmented regularity for the density, when the pressure is not present in the equation. Here, one can obtain a bound
    $$
    \sqrt{\rho}\in L^\infty(0, T; W^{1,2}(\R^d)).
    $$
    Since it is obtained via testing the momentum equation by $\nabla_x\log\rho$, the additional estimate in our case is for the term
    \begin{equation*}
        \begin{split}
            \left|\int_0^T\int_{\R^d\times\R^d}\rho(t, x)\nabla_x(\log\rho(t, x))D^2K(x-y)(u(t, x) - u(t, y))\rho(t, y)\diff y\diff x\diff t\right|.
        \end{split}
    \end{equation*}
    Similarly as in point 2, we may use Young's inequality and mass conservation to obtain
    \begin{equation*}
        \begin{split}
            &\left|\int_0^T\int_{\R^d\times\R^d}\rho(t, x)\nabla_x(\log\rho(t, x))D^2K(x-y)(u(t, x) - u(t, y))\rho(t, y)\diff y\diff x\diff t\right|\\
            & \phantom{=}= 2\left|\int_0^T\int_{\R^d\times\R^d}\sqrt{\rho(t, x)}\nabla_x\sqrt{\rho(t, x)}D^2K(x-y)(u(t, x) - u(t, y))\rho(t, y)\diff y\diff x\diff t\right|\\
            &\leq \|D^2K\|_\infty\int_0^T\int_{\R^d}|\nabla_x\sqrt{\rho}|^2\diff x\diff t + \|D^2K\|_\infty\int_0^T\int_{\R^d}\rho|u|^2\diff x\diff t\\
            &\phantom{=}+\|D^2K\|_\infty\left(T +\int_0^T\int_{\R^d}|\nabla_x\sqrt{\rho}|^2\diff x\diff t\right)\left(T + \int_0^T\int_{\R^d}\rho|u|^2\diff x\diff t\right).
        \end{split}
    \end{equation*}
    Comparing with the left-hand side of \eqref{ineq:maja_bresch_dejardins}, the inequality above, together with the bound on $\|\rho|u|^2\|_{L^\infty_tL^1_x}$, allows us to use Gr\"{o}nwall's inequality and deduce the Bresch--Dejardins estimate. At this point in the approximation, we get the convergence of the approximating sequence in $C([0, T]; L^{3/2})$ strongly for the density, in $L^2(0, T; H^1)$ strongly and in $L^2(0, T; H^2)$ weakly for the square root of the density, in $L^2(0, T; L^{3/2})$ strongly for the momentum, and in $L^2(0, T; L^2)$ weakly for the velocity (cf. \cite[Lemma 4.1, 4.3]{mucha2025construction}).
    
    \item The fourth difference  is in obtaining the so-called Mellet--Vasseur estimate (cf. \cite[Section 5, Lemma 5.1]{mucha2025construction}). The idea inside \cite{mellet2007onthebarotropic} is quite similar to the classical proof of the existence of weak solutions to compressible Navier--Stokes by Lions \cite{lions1998mathematical2}, where the augmented density regularity is used to improve the possible estimates on the velocity field. After obtaining this estimate, one can converge with $\kappa, r_0, r_1\to 0^+$ in \eqref{eq:appendix_approximate_system}. We again focus only on the non-local term, but we need some technical definitions either way. Let us define cut-off functions for density
    $$
    \phi^0_m(\rho) = 1,\text{ for } \rho > \frac{1}{m},\quad \phi^0_m(\rho) = 0,\text{ for }\rho < \frac{1}{2m}, \quad |(\phi^0_m)'| < 2m,
    $$
    and
    $$
    \phi^\infty_k(\rho) = 1,\text{ for }\rho < k,\quad \phi^\infty_k(\rho) = 0,\text{ for }\rho >2k,\quad |(\phi^\infty_k)'| <\frac{2}{k}.
    $$
    Then, we can set
    $$
    v_{m,k} = \phi_{m,k}(\rho)u,\qquad \phi_{m,k}(\rho) = \phi^0_m(\rho)\phi_k^\infty(\rho),
    $$
    although we shall skip the subscripts $m, k$ in the following explanation. Moreover, we fix the following approximations of the function $F(z) = \frac{1 + z^2}{2}\ln(1 + z^2)$
    \begin{align*}
        F_n(z) = \left\{\begin{array}{ll}
            \frac{1 + z^2}{2}\ln(1 + z^2),\quad z \leq n \\
            \left(nz + \frac{1 - n^2}{2}\right)\ln(1 + z^2),\quad z > n
        \end{array}\right.
    \end{align*}
    and for its derivative
    \begin{align*}
        \psi_n(z) = \frac{1}{z}F'_n(z) = \left\{\begin{array}{ll}
            1 + \ln(1 + z^2),\quad z \leq n \\
            \frac{n}{z}\ln(1 + z^2) + \frac{2nz + 1 - n^2}{1 + z^2},\quad z > n.
        \end{array}\right.
    \end{align*}
    Functions above are subject to the bounds 
    $$
    F_n(z) \leq C_n|z|^{1 +\delta,}\qquad \psi_n(z)z \leq C_n|z|^\delta,
    $$
    for some $C_n >0$ and $\delta\in (0, 1)$, and
    \begin{align}\label{bounds:bound_on_F_n_independent}
    F_n(z) \leq C + C|z|^{2 + \delta},\qquad \psi_n(z)z\leq C + C|z|^{1+\delta},
    \end{align}
    for some $C>0$. With those definitions at hand, we define a test function $\Phi = \xi(t)\psi_n(|v|)v \phi(\rho)$, for fixed non-negative $\xi\in C_c^\infty(0, T)$, and test the momentum equation with it. Then, we aim to estimate
    \begin{align*}
        \left|\int_0^T\int_{\R^d\times\R^d}\rho(t, x)\Phi(t, x)D^2K(x-y)(u(t, x) - u(t, y))\rho(t, y)\diff y\diff x\diff t\right|.
    \end{align*}
    To do so, we note that $F_n$ satisfies (cf. \cite[Proposition 5.7]{mucha2025construction})
    $$
    zF'_n(z) \lesssim F_n(z),\qquad F_n^*(F_n'(z)) \lesssim F_n(z),
    $$
    hence, with the use of mass conservation, and the Fenchel--Young inequality
    \begin{equation}\label{ineq:mellet_vasseur_ineq_for_bound_1}
    \begin{split}
        &\left|\int_0^T\int_{\R^d\times\R^d}\rho(t, x)\Phi(t, x)D^2K(x-y)(u(t, x) - u(t, y))\rho(t, y)\diff y\diff x\diff t\right|\\
        &\leq \|D^2K\|_\infty\int_0^T\int_{\R^d}\xi(t)F_n'(|v|)\phi(\rho)\rho(t,x)\int_{\R^d}\rho(t, y)|u(t,y)|\diff y\diff x\diff t\\
        &\phantom{=} + \|D^2K\|_\infty\int_0^T\int_{\R^d}\xi(t)F_n'(|v|)\phi(\rho)\rho|u|\diff x\diff t\\
        &\leq \|D^2K\|_\infty\int_0^T\int_{\R^d}\xi(t)\rho(t,x)\left(F_n^*(F_n'(|v|)) + F_n\left(\int_{\R^d}\rho(t, y) u(t, y)\diff y\right)\right)\diff x\diff t\\
        &\phantom{=}+\|D^2K\|_\infty\int_0^T\int_{\R^d}\xi(t)|v|F_n'(|v|)\rho(t, x)\diff x \diff t\\
        &\leq \|D^2K\|_\infty\int_0^T\int_{\R^d}\xi(t)\rho(t,x)\left(F_n^*(F_n'(|v|)) + F_n\left(\int_{\R^d}\rho(t, y) u(t, y)\diff y\right)\right)\diff x\diff t\\
        &\phantom{=}+C\|D^2K\|_\infty\int_0^T\int_{\R^d}\xi(t)F_n(|v|)\rho(t, x)\diff x \diff t.
    \end{split}
    \end{equation}
    Using the bound \eqref{bounds:bound_on_F_n_independent}, and the already shown energy bounds, we estimate
    \begin{equation}\label{ineq:mellet_vasseur_ineq_for_bound_2}
        \begin{split}
            &\int_0^T\int_{\R^d}\xi(t)\rho(t,x)\left(F_n^*(F_n'(|v|)) + F_n\left(\int_{\R^d}\rho(t, y) u(t, y)\diff y\right)\right)\diff x\diff t\\
            &\leq \int_0^T\int_{\R^d}\xi(t)\rho(t, x)(F_n(|v|) + C + C\|\rho u\|_{L^\infty_tL^1_x}^{2+\delta})\diff x\diff t\\
            &\leq \int_0^T\int_{\R^d}\xi(t)\rho(t, x)(F_n(|v|) + C + C\|\rho |u|^2\|_{L^\infty_tL^1_x}^{\frac{2+\delta}{2}})\diff x\diff t.
        \end{split}
    \end{equation}
    Combining \eqref{ineq:mellet_vasseur_ineq_for_bound_2} with \eqref{ineq:mellet_vasseur_ineq_for_bound_1} we are set to use the weak version of Gr\"{o}nwall's lemma, which we recall below for the convenience of the reader.
    \begin{prop}[Lemma B.1, \cite{mucha2025construction}]
        Let $f\in L^1(0, T)$ satisfy
        $$
        -\int_0^T\xi'(t)f(t)\diff t\leq\int_0^T\xi(t)(a f(t) + b(t))\diff t,
        $$
        for any $\xi\in C_c^\infty(0, T)$, $\xi\geq 0$, a constant $a\geq 0$, and a non-negative function $b\in L^1(0, T)$. Then, for almost all $0\leq s < t < T$ we have
        $$
        f(t) \leq f(s) e^{a(t-s)} + \int_s^t e^{a(t-\tau)}b(\tau)\diff\tau.
        $$
    \end{prop}
    Thus, we get the needed estimate on 
    $$
    \int_{\R^d}\rho(t, x)F(|u|)\diff x.
    $$
    Note, that in our case this technique does not provide us with the estimate on
    $$
        \sup_{t\in [0, T]}\int_{\R^d\times \R^d}F(|x - y|)\rho(t, x)\rho(t, y)\diff x\diff y,
    $$
    which is present in \eqref{ineq:maja_mellet_vasseur}. At the end of this approximation step, we obtain the convergence in $C([0, T]; L^{3/2})$ strongly for the density, in $L^2(0, T; L^p)$, $(p < 3/2)$ strongly for the momentum, and in $L^2(0, T; L^2)$ strongly for the square root of the density multiplied by the velocity (cf. \cite[Lemma 6.1]{mucha2025construction}).
    
    \item The part of the proof devoted to expansion of the torus to the whole space is actually the same as in \cite{mucha2025construction}.  We only note that all of the bounds on the new nonlocal term from previous points are uniform with respect to the size of the domain. Hence, one can obtain the same convergences as in the step before when letting the size of torus go to $+\infty$.
    \item At last, we need to explain how to obtain the second moment bounds on density. We note that as approximation procedure is done by the expansion of the torus, it is possible to test the continuity equation in \eqref{eq:appendix_approximate_system} (after the convergence with $\varepsilon\to 0^+$), by $|x|^2$. Then, using Young's inequality
    \begin{align*}
        &\int|x|^2\rho(t, x)\diff x - \int|x|^2\rho_0(x)\diff x = 2\int_0^t\int\rho u\cdot x\diff x \\
        &\leq \int_0^t\int|x|^2\rho(s, x)\diff x\diff s + \int_0^t\int\rho(s, x)|u(s, x)|^2\diff x\diff s,
    \end{align*}
    which by classical Gr\"{o}nwall's inequality grants us desired bounds.
\end{enumerate}

\textbf{Acknowledgment.}
For the purpose of open access, the author has applied a Creative Commons Attribution (CC BY) licence to any Author Accepted Manuscript version arising from this submission.

\bibliographystyle{abbrv}
\bibliography{Nonlocal_AR}
\end{document}